\documentclass[12pt,reqno]{amsart}
\usepackage{amssymb, amsmath, amsfonts, amsthm, graphics}
\usepackage[hmargin=1 in, vmargin = 1 in]{geometry}

\RequirePackage[dvipsnames,usenames]{color}

\usepackage{mathrsfs}
\usepackage{mathtools}
\usepackage{tikz-cd} 
\usetikzlibrary{matrix, calc, arrows} 
\usepackage[hyphens]{url}
\usepackage{breakurl}

\usepackage{hyperref}
\usepackage[all]{xy}
\usepackage{marginnote}

\usepackage[active]{srcltx}
\usepackage{calc,amscd, eucal,ulem,stmaryrd}
\usepackage{alltt}

\usepackage{soul}
\usepackage{hyperref}

\input{kmacros3.sty}
\input{xy}
\xyoption{all}
\usepackage{tikz}
\usetikzlibrary{decorations.markings}
\usetikzlibrary{cd}

\usepackage{mabliautoref}
\usepackage{bm}
\usepackage{comment}
\usepackage[scr=boondox, cal=esstix]{mathalpha}

\newcommand{\isom}{\cong} 

\theoremstyle{definition}
\numberwithin{equation}{section}

\newcommand{\vol}{\mathrm{vol}}
\newcommand{\Amp}{\mathrm{Amp}}

\newcommand{\sumC}{\sum_{m=C_2}^{\frac{C_1p^e}{\|L\|}}}

\newcommand{\CC}{\mathbb{C}}

\newcommand{\NN}{\mathbb{N}}

\newcommand{\QQ}{\mathbb{Q}}
\newcommand{\RR}{\mathbb{R}}

\newcommand{\cC}{\mathcal{C}}
\newcommand{\cD}{\mathcal{D}}

\newcommand{\cF}{\mathcal{F}}
\newcommand{\cG}{\mathcal{G}}

\newcommand{\cL}{\mathcal{L}}

\newcommand{\fm}{\mathfrak{m}}

\newcommand{\fa}{\mathfrak{a}}
\newcommand{\fb}{\mathfrak{b}}

\newcommand{\Fe}{F^{e}_{*}}

\newcommand{\fad}{\mathfrak{a}_{\bullet}}
\newcommand{\tfad}{\tilde{\mathfrak{a}}_\bullet}
\newcommand{\s}{\mathscr{s}}
\newcommand{\FA}{\alpha_{F}}

\newcommand{\tr}{\mathrm{tr}^e}

	\title{The $F$-signature function on the big cone}

    \author{Seungsu Lee}
    \address[S.~Lee]{Department of Mathematics\\University of Michigan\\Ann Arbor, 
		MI 48109-1043\\USA}
    \email{\href{mailto:starlee@umich.edu}{starlee@umich.edu}}
    
    \author{Suchitra Pande}
	\address[S.~Pande]{Department of Mathematics\\University of Utah\\Salt Lake City, 
		UT 84112\\USA}
	\email{\href{mailto:suchitra.pande@utah.edu}{suchitra.pande@utah.edu}}
    \thanks{Pande was partially supported by the NSF grants \#1952399, \#1801697 and \#2101075}
    
\begin{document}

\maketitle
\begin{abstract}
We define and study the $F$-signature of big divisors on projective varieties in positive characteristics. For a normal projective variety $X$, we show that the $F$-signature function defines a continuous function on the big cone of $X$, extending our previous results on the ample cone. By studying the the Frobenius-alpha invariant on the big cone, we prove that the positivity of the $F$-signature of any big $\mathbb{R}$-divisor on $X$ is equivalent to the global $F$-regularity of $X$. As applications, we prove a birational transformation rule for the $F$-signature in terms of the $F$-signature of certain graded ideal sequences. We also prove that in a reduction modulo $p$ situation, a characteristic independent lower bound for the $F$-signature of $X$ with respect to any ample divisor $L$ spreads out to a similar bound for arbitrary big divisors on $X$.
\end{abstract}

\section{Introduction}
The $F$-signature is a numerical measurement of singularities in positive characteristics via the asymptotic flatness of the Frobenius map. For an $F$-finite local domain $R$ of positive characteristic $p$, the $F$-signature is the limiting percentage of free summands (as an $R$-module) in $R^{1/p^e}$ as $e \to \infty$. This invariant was formalized in \cite{SmithVanDenBerghSimplicityOfDiff,HunekeLeuschkeTwoTheoremsAboutMaximal} and the defining limit was shown to exist in \cite{TuckerFSigExists}. Furthermore, \cite{AberbachLeuschke} characterized strongly $F$-regular rings in terms of the positivity of their $F$-signatures. The numerical value of the $F$-signature provides subtle information about strongly $F$-regular singularities, including bounds on the sizes of the local fundamental groups \cite{carvajal-rojas_fundamental_2016}, divisor class groups \cite{PolstramaximalCM, MartinTorsionDivisors} and their local volumes.

Globally, the $F$-signature provides an interesting invariant of a pair $(X,L)$ consisting of a projective variety $X$ and an ample divisor $L$ over $X$ by considering the singularities of the section ring $S(X,L):= \bigoplus_{m\in \mathbb{Z}}H^0(X,\cO_X(mL))$. The $F$-signature (at the homogeneous maximal ideal) of $S(X,L)$ is a subtle variant of the volume of $L$ and helps detect the global $F$-regularity of $X$ (which in turn implies the Kodaira vanishing theorem for $X$).
Inspired by this analogy with the volume, our previous work \cite{LeePandeFsigAmpCone} defined and studied an $F$-signature \emph{function} $\s_X: \Amp (X) \to \mathbb{R}$ on the ample cone of $X$, that assigns: to any ample class $L$, the $F$-signature of its section ring $S(X,L)$; and satisfies the scaling property $s_X (\lambda L) = \frac{s_X (L)}{\lambda}$ for any $\lambda >0$. Moreover, we showed that the $F$-signature function is locally Lipschitz continuous at all real ample divisors and extends continuously to the non-zero part of the boundary of $\Amp (X)$.

\subsection{The $F$-signature of big divisors} The goal of this paper is to extend the $F$-signature function from the ample cone to the big cone of a normal projective variety $X$. There are two key difficulties in this extension:
\begin{itemize}
    \item The section ring $S(X,L)$ with respect to a big divisor $L$ may not be finitely generated, and thus, previous results on the existence of the defining limit \cite{TuckerFSigExists, PolstraTuckerCombinedApproach} do not apply. We note though that the finite generation of $S(X, L)$ is expected to hold for big divisors on globally $F$-regular varieties (see \cite{DattaSchwedeTuckerFiniteGenerationofSplitFregularMonoidAlgebras}).
    \item The section ring of a big divisor $L$ is invariant under pulling back along a proper birational map $\pi : Y \to X$. So, even when the section $S(X,L)$ is finitely generated, its $F$-signature is not intrinsic to $X$.
\end{itemize}

We address both these issues by exploiting the grading on $S(X,L)$ and measuring Frobenius splittings of the linear systems $H^0 (X, \cO_X(mL))$ individually instead of the entire section ring. More precisely, we define:
\begin{dfn}
    The $F$-signature of a big Cartier divisor $L$ on a normal projective variety $X$ over an algebraically closed field of characteristic $p>0$ is defined to be the limit:
    \[s_X(L):= \lim_{e\to\infty} \frac{1}{p^{e{(\dim X+1)}}}\sum_{m=0}^\infty\dim_k\frac{H^0(X, \cO_X (mL))}{I_e(mL)},\]
    where the Frobenius-splitting subspaces $I_e (mL)$ are defined as 
    \[I_e(mL):= \{f\in H^0(X,\cO_X (mL))\mid \varphi(F^e_*f) = 0 \text{ for all } \varphi \in  \Hom_{\cO_X}(F^e_*( \cO_X(mL)),\cO_X)) \}.\]
\end{dfn}

When $L$ is ample, $\s_X(L)$ equals the $F$-signature of the corresponding section ring $S(X, L)$ as showed in \cite[Lemma 4.6]{LeePandeFsigAmpCone}. With this definition, we are able to continuously extend the $F$-signature function to the big cone of $X$:


\begin{theoremA*}
Let $X$ be a normal projective variety over an algebraically closed field of positive characteristic and $L$ be a big Cartier divisor on $X$. Then,
\begin{enumerate}
\item \textnormal{(\autoref{MainthmA})} The limit defining $s_X(L)$ exists.
\item \textnormal{(\autoref{Pn:TransformationRuleFsig})}  We have a transformation rule $s_X(L^n)= \frac{1}{n}s_X(L)$ for any $n \geq 1$.
\item \textnormal{(\autoref{thm:FsigLocallyLipschitz})} The assignment $L \mapsto \s_X (L)$ extends to a well-defined, locally Lipschitz continuous function on the (real) big cone of $X$.

\end{enumerate}
\end{theoremA*}

\subsection{Positivity and comparison to the volume function} In the case of ample divisors, the $F$-signature $s_X(L)$ is positive for some (equivalently, any) ample $L$ if and only if $X$ is globally $F$-regular \cite{SchwedeSmithLogFanoVsGloballyFRegular}. The same statement is true for the $F$-signature of big divisors as well:
\begin{theoremB*}[\autoref{thm:positivityofFsigbig}]
Let $X$ be a normal projective variety over an algebraically closed field of positive characteristic. Then, $X$ is globally $F$-regular (\autoref{dfn:globalFregularity}) if and only if the $F$-signature $s_X (L)$ with respect to some (equivalently, any) big Cartier divisor $L$ is positive.  
\end{theoremB*}

 To study the positivity of $s_X(L)$, we study the Frobenius-alpha invariant (\autoref{Falphadfn}) on the big cone. The $\FA$-invariant is a positive characteristic version of Tian's alpha invariant, first defined in \cite{PandeFrobeniusVersionTiansAlphaInvariant}. We extend the theory of the $\FA$-invariant to big divisors in \autoref{Section:AlphaInvariant} and apply its continuity properties to the $F$-signature. The advantage of working with the $\FA$-invariant is due to its monotonicity properties on the big cone (\autoref{Prop:addingeffective}).
 Using this, we are able to establish effective comparisons between the $F$-signature function and the volume function.

\begin{theoremC*} \label{thm:theoremC}
Let $X$ be a normal projective variety over $k$. Then, \begin{enumerate}

\item\textnormal{(\autoref{Prop:Fsigandalphacomparison})} Let $L$ be a big Cartier divisor on $X$. We have the following comparison between the $F$-signature, volume function and the $\FA$-invariant:\[ \s_X(L) \geq \frac{\vol(L) \, \FA(X, L)^{d+1}}{(d+1)!}.\]
\item\textnormal{(\autoref{thm:positivityforbig})} Suppose further that $X$ is globally $F$-regular. Then, for any big $\bR$-divisor $L$, the $F$-signature $s_X(L)$ is positive. Moreover, for any compact subset $\kappa$ that is contained in the big cone, there exists a constant $C>0$ (depending on $\kappa$) such that $\s_X(L) \geq C \vol(L)$ for all $L  \in \kappa$.
\end{enumerate}
\end{theoremC*}

We note that Part (b) of Theorem C is new even for ample $\RR$-divisors.

\subsection{Transformation rule under birational maps} As noted earlier, the section ring $S(X, L)$ remains invariant when $L$ is pulled-back along a proper and birational map $\pi: Y \to X$. However, our definition of the $I_e$ subspaces are intrinsic to $X$ and naturally track changes in the birational models of $X$. To state the precise results, we extend the notion of the $F$-signature to the context of \emph{graded sequences of coherent ideals} of $\cO_X$ (\autoref{def:fsigofidealsequence}). We develop the theory of global $F$-regularity and the $F$-signature in this context using the formalism of \emph{Cartier algebras} (\autoref{def:Cartiersubalgebraassociatedtogradedsequence}).

\begin{theoremD*} [\autoref{Thm:transformationRuleProperBirational}]  \label{thm:theoremD}
   Let $X$ be a globally $F$-regular projective variety and $L$ be a big and semi-ample divisor over $X$. Suppose $\pi: X \to Y$ is the birational contraction defined by $|mL|$ for some $ m \gg 0$ sufficiently divisible and $L_0$ denote the ample divisor on $Y$ such that $\pi ^* L_0 \sim L$. Then, we have
    \begin{enumerate}
        \item Suppose $K_Y $ is $\QQ$-Cartier and $K_{Y|X} = K_Y - \pi^* K_X \leq 0$. Then,  $\s_X (L) = \s_Y (L_0) $.

        \item Suppose $K_Y $ is $\QQ$-Cartier and we have $K_{Y|X} = E \geq 0$ for some effective $\QQ$-divisor $E$. Then, we have a natural graded sequence of ideal sheaves over $Y$ defined by
        \[ \fa_m := \pi_* \cO_X (\lceil -mE \rceil ).  \]
        Moreover, we have the following transformation rule for the $F$-signature:
        \[ \s_X (L) = \s_Y (\fad, L_0). \]
        
        \item For any effective $\QQ$-divisor $\Delta_X $ such that $(X, \Delta_X)$ is globally $F$-regular and $-K_X - \Delta_X  \sim _\QQ L$, we have:
        \begin{itemize}
            \item setting $\Delta_Y = \pi_* \Delta_X $, the pair $(Y, \Delta_Y)$ is globally $F$-regular.

            \item the divisor $- K_Y - \Delta_Y$ is $\QQ$-Cartier and $ - K_Y - \Delta_Y \sim _\QQ L_0 $.

            \item We have 
            \[ \s_X (\Delta_X, L) = \s_Y (\Delta_Y, L_0).\]
        \end{itemize}
        
    \end{enumerate}
\end{theoremD*}

This theorem is motivated by several explicit computations in \cite{VonKorffthesis, LeePandeFsigAmpCone, LeePandeSimpsonHilbertkunzmultiplicityfsignaturedisagree} suggesting that when we blow-up a smooth projective variety $X$ along a smooth center $Z \subset X$, then the $F$-signature of the pull-back of any ample divisor on $X$ measures the singularities of the pair $(X, Z)$ (see also \cite{LakshmibaiFrobSplittingsandBlowups}). Theorem D is a generalization of this observation to more general birational maps.


\subsection{Application to the positivity of limit $F$-signature.} There has been a recent surge of interest in the studying the behavior of the $F$-signature under the reduction mod $p$ process; see \cite{CaminataShidelerTuckerZermanFsignaturefunctionsdiagonalhypersurfaces,TakagiYamaguchiUniformPositivityFsignatureReduction,BrosowskyCoskunPandeTuckerlimitfsignaturefunctionstwovariable,LiuPandepositivitylimitfsignature}. However, the invariant remains difficult to compute in general. As an application of the theory we developed above, we prove a uniform spreading-out principle for lower bounds of $F$-signature function: if the $F$-signature admits a characteristic–independent lower bound for one big divisor, then analogous characteristic–independent lower bounds exist for every big divisor.

\begin{theoremE*}[\autoref{thm:LowerBoundsSpreadout}]
  Let $X$ be a projective variety of globally $F$-regular type over $\CC$. Suppose we have an ample divisor $L$ and a constant $C:= C(L)$ such that for all $p \gg 0$, we have the lower bound $ \s (X_s, L_s) \geq C $ for all $p \gg 0$. Then, for any ample divisor $L'$ on $X$, after possibly enlarging $A$, there exists a constant $C' := C' (L')$ such that $\s (X_s, L'_s) \geq C'$ for all closed points $s \in \Spec (A)$ of characteristic $p$ and all $p \gg 0$.
\end{theoremE*}

\subsection*{Organization of the paper} 
\autoref{section:Fsigofbig} is dedicated to the detailed study of the $F$-signature of big divisors and the $F$-signature function on the big cone. We study the $\FA$-invariant in \autoref{section:alphainvarinat} and apply it to the positivity of the $F$-signature. We extend the theory of global $F$-regularity and the $F$-signature to graded sequences of ideals in \autoref{section:gradedsequences}. In \autoref{section:pairsandtriples} we make some remarks about extensions of the $F$-signature function to the case of pairs $(X, \Delta)$ where $\Delta$ is a $\QQ$-divisor on $X$, and in addition, to the case of triples $(X, \Delta, \fad)$ where $\fad$ denotes a graded sequence of coherent ideals over $X$. We apply these extensions to prove a comprehensive transformation rule for the $F$-signature under proper and birational maps in \autoref{section:transformationrule}. Finally, the application to the positivity of the limit $F$-signature is presented in \autoref{section:positivity}.
 
{\bf Acknowledgements:} We would like to thank Anna Brosowsky, Benjamin Baily, Sung Rak Choi, Havi Eller, Yuchen Liu, Karl Schwede, Austyn Simpson, Karen Smith, Olivia Strahan, Shunsuke Takagi and Kevin Tucker for valuable suggestions and discussions. We also thank the mathematics departments at the University of Utah and the University of Michigan for their hospitality and support during the various visits by the authors.

\section{Preliminaries}
In this section, we discuss the fundamental backgrounds we used in later chapters. We begin with divisors and the volume function on divisors that describe the asymptotic growth of spaces of sections of associated invertible sheaves. We then recall the relation between invertible sheaves and section modules on a section ring of projective varieties. Finally, we recall the global $F$-regularity and the $F$-signature function of ample divisors.

\subsection{Divisors and Cones}
Let $X$ be a normal projective variety over an algebraically closed field $k$. Let $N^1(X)$ denote the N\'eron-Severi space of $X$ defined as the vector space of Cartier divisors modulo numerical equivalence. We denote $N^1(X)_\bQ := N^1(X)\otimes \bQ$ and $N^1(X)_\bR := N^1(X)\otimes \bR$ for the space of $\bQ$- and $\bR$-classes respectively. Throughout the paper, we denote a divisor $D$ by which we mean a Cartier divisor on $X$ unless otherwise specified.

Now, we recall ample cone and pseudo-effective cone, which are going to be the domain of the $F$-signature and the Frobenius alpha invariant that we will discuss in later chapters. 

\begin{dfn}\cite{LazarsfeldPositivity1}
The \emph{ample cone} $\Amp(X)\subseteq N^1(X)_\bR$ is the open convex cone of all ample $\mathbb{R}$-divisor classes. The \emph{Nef cone} $\mathrm{Nef}(X)$ is the closure of the ample cone consisting of numerically effective $\mathbb{R}$-divisor classes. We denote the \emph{cone of effective $\bR$-divisors} by $\mathrm{Eff}(X)$ and its closure $\overline{\mathrm{Eff}}(X)$ is called the \emph{pseudo-effective cone}.
\end{dfn}

\begin{dfn}
A divisor $D$ is $\emph{big}$ if $D$ is in the interior of the pseudo-effective cone. 
\end{dfn}

The set of all big divisor classes forms the \emph{big cone} $\mathrm{Big}(X) \subseteq N^1(X)_\mathbb{R}$. Recall that the big cone is an open convex cone containing $\Amp(X)$. Throughout the paper, we will denote $\mathrm{Big}_\bQ(X):= \mathrm{Big}(X) \cap N^1(X)_\bQ$, $\Amp_\bQ(X):= \mathrm{Amp}(X)\cap N^1(X)_\bQ$ be corresponding subcones of $N^1(X)_\bQ$ consisting of rational divisor classes.

\subsection{Volume function on big divisors.}

Let $X$ be a normal projective variety over $k$ and $L$ be an invertible sheaf on $X$. Recall that the volume of $L$ is defined as the limit
\[ \vol(L) := \lim_{m \to \infty} \frac{\dim_k  H^0 (X, L^m) }{m^d /d!},  \]
where $d$ denotes the dimension of $X$. This limit exists by \cite{LazarsfeldMustataConvexbodies}. The volume of a divisor $D$ is the volume of the corresponding invertible sheaf $\vol(\cO_X(D))$. Now, for any $n\in \mathbb{N}$, $\vol(nD) = n^d\vol(D)$. For a $\bQ$-Cartier divisor $D$, choose $r>0$ such that $rD$ is Cartier. Then, $\vol(D):= \vol(rD)/r^d$. For more details, see \cite{LazarsfeldPositivity1}. Throughout the paper, we will denote $\vol(D)$ for the volume of the divisor. Note that $D$ is big if and only if $\vol(D)>0$.

\begin{lem} \label{lem:sumvol}
  Let $X$ be a $d$-dimensional normal projective variety $L$ be an invertible sheaf on $X$. Let $P(N)$ denote the function $\sum _{m = 0 } ^N  \dim_k H^0 (X, L^m)$. Then we have
  \[ \lim_{N \to \infty} \frac{P(N)}{N^{d+1}/(d+1)!} = \vol(L).  \]
\end{lem}
\begin{proof}
    Fix an $\varepsilon >0$. Then, by the definition of the volume function, there exists a constant $M > 0$ (depending on $\varepsilon$) such that for any $m\geq M$,
    \[\left(\vol(L)-  \frac{\varepsilon}{2}\right) \frac{m^d}{d!}  < \dim_k H^0 (X, L^m) < \left(\vol(L)+ \frac{\varepsilon}{2}\right) \frac{m^d}{d!} .\]
    Now, we take $N_0 \in \mathbb{N}$ with $M< N_0$ a sufficiently large number such that $P(M) \leq \frac{\varepsilon}{2} N_0^{d+1} /(d+1)!$. Now, for any $N\geq  N_0$, \[\begin{split} P(N)&=P(M) + \sum_{i=M+1}^N \dim_k H^0(X,L^i) \\ & < \frac{\epsilon N^{d+1}}{2(d+1)!} + \sum_{i=M+1}^N \left(\vol(L) +\frac{\epsilon}{2}\right)\frac{i^d}{d!} \\ &< \frac{\epsilon N^{d+1}}{2(d+1)!} + \left(\vol(L) +\frac{\epsilon}{2}\right) \frac{N^{d+1}}{(d+1)!} +O(N^d)\\ &= (\vol (L) + \epsilon) \frac{N^{d+1}}{(d+1)!}+O(N^d).\end{split}\]
    With a similar argument as above, one obtains a lower bound as follows: \[P(N)>(\vol(L) -\epsilon) \frac{N^{d+1}}{(d+1)!}+O(N^d).\] Therefore, for all $N \gg 0$, we have
    \[ (\vol(L) - \varepsilon) \frac{N^{d+1}}{(d+1)!} +O(N^d) < P (N) <  (\vol(L) + \varepsilon) \frac{N^{d+1}}{(d+1)!}+O(N^d).  \]
    This implies that we have
    \[ \vol(L) - \varepsilon \leq \liminf_{N \to \infty} \frac{P(N)}{N^{d+1}/(d+1)!} \leq \limsup_{N \to \infty} \frac{P(N)}{N^{d+1}/(d+1)!} \leq \vol(L) + \varepsilon.    \]
    Since this is true for arbitrary $\varepsilon >0$, we see that the limit $\lim_{N \to \infty} \frac{P(N)}{N^{d+1}/(d+1)!} $ exists and is equal to $\vol(L)$.
\end{proof}

\subsection{Section rings and Cones}
The $\NN$-graded rings we will be interested in arise as the section rings of projective varieties over $k$ with respect to some ample divisor. In this subsection, we will review this and other related constructions.
  	
 		\begin{dfn}[Section Rings and Modules] \label{sectionringdfn}
		Let $X$ be a projective variety over $k$, $\cL$ an ample invertible sheaf on $X$ and $\cF$ a coherent sheaf on $X$. Then the $\NN$-graded ring $S$ defined by
		$$ S =	S(X, \cL) :=  \bigoplus _{n \geq 0} H^{0}(X, \cL^{n})	$$
		is called the \emph{section ring} of $X$ with respect to $\cL$. The affine scheme $\Spec(S)$ is called the \emph{cone over $X$} with respect to $\cL$. The \emph{section module} of $\cF$ with respect to $\cL$ is a $\ZZ$-graded $S$-module $M$ defined by
		$$M = M(\mathcal{F}, \cL) := \bigoplus _{n \in \ZZ} H^{0}(X, \cF \otimes \cL^{n}) 	.$$
		Similarly, the sheaf corresponding to $M$ on $\Spec(S)$ is called the \emph{cone over $\cF$} with respect to $\cL$.
	\end{dfn}

\begin{lem}\label{lem.SectionRingVSCoherentSheaf} Let $X$ be a projective variety over $k$ and $\cL$ an ample invertible sheaf over $X$. 
\begin{enumerate}
    \item The section ring $S$ of $X$ with respect to $\cL$ is always finitely generated over $k$ and hence, is Noetherian.  If $X$ is normal, then the section ring is also characterized as the unique normal $\NN$-graded ring $S$ such that $\Proj(S)$ is isomorphic to $X$ and the corresponding $\cO_X(1)$ is isomorphic to $\cL$.
    \item The section module of any coherent sheaf over $X$ with respect to $\cL$ is finitely generated over $S$. It is also characterized as the unique saturated (with respect to the homogeneous maximal ideal) $S$-module $M$ such that the associated coherent sheaf $\tilde{M}$ on $X$ is isomorphic to $\cF$. 
    \item For two coherent sheaves $\cF$ and $\cG$, we have a natural isomorphism:
    $$ \Hom_{\cO_X}(\cF, \cG)  \isom \Hom^{\text{gr}}_{S}(M(\cF, \cL), M(\cG, \cL))     $$
    where $\Hom ^{\text{gr}} _S(\, , \, )$ denotes the set of grading preserving $S$-module maps between two graded $S$-modules. 
    \end{enumerate}
\end{lem}
\begin{proof}
See \cite[\href{https://stacks.math.columbia.edu/tag/0BXF}{Tag 0BXF}]{stacks-project}.
\end{proof}

\subsection{Globally $F$-regular varieties and $F$-signature of ample divisors} Let $X$ be a normal variety over an algebraically closed field $k$ of positive characteristic $p$. Let $F:X\to X$ be the Frobenius morphism raising the power by $p$ on the structure sheaf. We denote $F^e = F\circ F\circ\dots \circ F  $ as the $e$-th iterated Frobenius morphism.

\begin{dfn} \label{dfn:globalFregularity}
Let $X$ be a normal variety over an algebraically closed field $k$ of positive characteristic. $X$ is globally $F$-regular if for any effective Weil divisor $D$ on $X$, there exists some integer $e>0$ such that the natural map $\cO_X\to F^e_*\cO_X(D)$ splits as a map of $\cO_X$-modules.
\end{dfn}

\begin{dfn} \label{Iedfn}
Let $X$ be normal projective variety over $k$ and $D$ denote any Weil divisor on $X$. Then, for any integer $e \geq 1$, we define a $k$-vector subspace $I_e(D)$ of $H^0 (D):= H^0 (X, \cO_X(D)) $ as \[I_e(D):= \{f\in H^0(D)\mid \varphi(F^e_*f) = 0 \text{ for all } \varphi \in  \Hom_{\cO_X}(F^e_*( \cO_X(D)),\cO_X)) \}.\]
\end{dfn}

Now, we recall some technical properties of $I_e(D)$.

\begin{lem} \label{lem:Ieeffective}
    Let $X$ be a normal, projective variety over $k$ and $D$ and $E$ be effective Weil divisors on $X$. Note that the effective divisor $E$ defines an inclusion $H^0 (X, \cO_X(D)) \subset H^0 (X, \cO_X(D+E))$. Then, for any $e \geq 1$, we have
    \[ I_e (D) \subset I_e (D+E). \]
    \end{lem}
    \begin{proof}
        See \cite[Lemma 4.12]{LeePandeFsigAmpCone}.
    \end{proof}

\begin{thm}\cite[Theorem 4.9]{LeePandeFsigAmpCone}\label{Thm:DegenSpaceEqualsWholeSpace}
Let $X$ be a normal projective variety over $k$ with $\dim(X)>0$. Fix a norm $\|\|$ on the N\'eron-Severi space $N^1_\bR(X)$. Let $L$ be an effective Cartier divisor such that $\|L\|\neq 0$. Then, there exists a constant $C$ depending only on $X$ and the norm such that for any effective Cartier $H$ on $X$, we have: 
\begin{enumerate}
\item $I_e(mL) = H^0(mL)$ for all $e\geq 1$ and $m> Cp^e/\|L\|$
\item For all $n>2\|H\|/\|L\|$, \[I_e(m(nL+H))=H^0(m(nL+H)) \text{ for all $e\geq 1$ and $m>\frac{Cp^e}{n\|L\|}$}.\] 
\end{enumerate}
\end{thm}

\begin{dfn}(\cite{VonKorffthesis},\cite{LeePandeFsigAmpCone})
    Let $X$ be a globally $F$-regular projective variety over $k$. Suppose that $d=\dim(X)>0$. Let $L$ be an ample $\mathbb{R}$-divisor on $X$. Then, the $F$-signature function $\s_X(L)$ is defined as \[\s_X(L) = \lim_{e\to\infty}\frac{1}{p^{e(d+1)}}\sum_{m=0}^\infty\dim_k \frac{H^0(mL)}{I_e(mL)}.\]
\end{dfn}

For globally $F$-regular varieties, the numerical equivalence implies the linear equivalence of divisors \cite[Theorem 3.4]{LeePandeFsigAmpCone}. Hence $s_X(L)$ is a well-defined function on $\mathrm{Amp}(X)$. For an integral ample divisor $L$ on $X$, the $F$-signature function $\s_X(L)$ equals the $F$-signature of the section ring of $X$ with respect to $L$ \cite{VonKorffthesis}, \cite{LeePandeFsigAmpCone}. Hence, it is positive if and only if $X$ is globally $F$-regular \cite{AberbachLeuschke}, \cite{SmithGloballyFRegular}. The above limit exists and the $F$-signature function extends continuously to $\mathrm{Nef}(X)\setminus\{0\}$ of $X$ \cite{LeePandeFsigAmpCone}.

\subsection{Restricting Weil divisors to normal, complete intersection subschemes} \label{restrictingWeildivisors} Let $X$ be a normal variety over $k$ and $D$ be a Weil divisor on $X$. Suppose $Y \subset X$ is a closed subvariety that is locally defined by a regular sequence in $X$. Assume that $Y$ is also normal. In this situation, we may ``restrict" the rank one reflexive sheaf $\cF:= \cO_X(D)$ on $X$ to a reflexive sheaf $\cF_Y$ on $Y$ as follows: Let $U$ be the regular locus of $Y$. Then there is an open subset $V \subset X_{\text{reg}}$ (where $X_{\text{reg}}$ denotes the regular locus of $X$) such that $V \cap Y = U$. This is possible because $Y$ is a complete intersection in $X$. Therefore, we may restrict $\cF$ to $V$ and then to an invertible sheaf on $U$, since $\cF|_V$ is invertible. Define $\cF_Y$ to be
            \[ \cF_Y :=  i_* (\cF|_U) \]
            where $i: U \to Y$ is the inclusion. Then, $\cF_Y$ is a rank one reflexive sheaf on $Y$ because $Y$ is normal and $U$ contains all the codimension one points of $Y$. Thus, we can write $\cF_Y$ as $\cO_Y (D_Y)$ for some Weil-divisor $D_Y$ on $Y$. Furthermore, if $\text{Supp}(D)$ does not contain $Y$, then since $Y$ is normal, hence integral, $D$ naturally restricts to a Cartier divisor $D_U$ on $U$ (given by restricting the equation for $D$) and we may take $D_Y$ to be the closure of $D_U$. It is also clear from the description of restriction that it commutes with addition of Weil-divisors (since the restriction of Cartier divisors on the regular locus commutes with addition).

            Now assume that $Y$ is a normal Cartier divisor on $X$. Then, restricting a Weil divisor $D$ on $X$ to $Y$ as described above, we note that we have a left exact sequence
            \[ 0 \to \cO_X(D - Y) \to \cO_X(D) \to \cO_Y (D_Y). \]
            Note that $X_{\text{reg}} \cap Y$ contains all the codimension one points of $Y$. Thus, pushing forward from $X_{\text{reg}}$, we obtain a left exact sequence
            \begin{equation} \label{restrictingtonormaldivisorexactsequence} 0 \to H^0 (X, \cO_X(D - Y)) \to H^0 (X, \cO_X(D)) \to H^0 (Y, \cO_Y(D_Y)). \end{equation}

 \section{The $F$-signature of big divisors} \label{section:Fsigofbig}

In this section, we define the $F$-signature function for big divisors and prove its existence. Let $k$ be an algebraically closed field of characteristic $p>0$. Throughout the section, $X$ will denote a normal projective variety over $k$ with $\mathrm{dim}(X)>0$. For the sake of simplicity of certain formulas, we abuse notation for the global sections $H^0 (X, \cF)$ of a sheaf $\cF$ computed on $X$ by dropping $X$ and simply write $H^0 (\cF)$, or even $H^0 (L)$ for a divisor $L$ when we mean $H^0 (X, \cO_X (L))$.

Consider a big Cartier divisor $L$ on $X$. Recall that for integers $e,m \geq 1$, we have the splitting subspace 
\[I_e(mL) = \{ s \in H^0(mL) \mid \varphi(F^e_* s) = 0 \text{ for all } \varphi \in \Hom_{\cO_X}(F^e_*\cO_X(mL), \cO_X)\}.\]

\begin{dfn} \label{dfn:Fsigbig} \label{Dfn.FsigBigDivisors}
    The $F$-signature of a $d$-dimensional normal projective variety $X$ with respect to a big Cartier divisor $L$ on $X$ is defined to be
    \begin{equation}\label{eqn:FsigFormula}
    \s _X (L) : = \lim_{e\to\infty}\frac{1}{p^{e(d+1)}}\sum_{m=0}^\infty\dim_k \frac{H^0(mL)}{I_e(mL)}\end{equation}
\end{dfn}

\subsection{Existence of the $F$-signature.}
    The main goal of this section is to show that the above limit exists. Note that the section ring $\bigoplus _{m \geq 0} H^0 (X, \cO_X (mL))$ may not be finitely generated and we do not assume this. Hence, \cite{TuckerFSigExists} does not directly apply here to guarantee the existence of the above limit.
    
\begin{thm} \label{thm:Fsigbigexists} \label{MainthmA}
    The limit \autoref{eqn:FsigFormula} defining the $F$-signature of a big divisor exists.
\end{thm}
\begin{rem}
The sum in the formula \autoref{eqn:FsigFormula} is finite for each $e \geq 1$. Indeed, by \autoref{Thm:DegenSpaceEqualsWholeSpace}, there exists a constant $C$ such that $H^0(mL) = I_e(mL)$ for all $m>Cp^e/\|L\|$.
\end{rem}

\begin{rem}
    If $L$ is ample, then \autoref{dfn:Fsigbig} coincides with the $F$-signature of the section ring $S (X, L)$ over $X$ \cite[Theorem 3.3]{LeePandeFsigAmpCone}.
\end{rem}
 
  Now, we proceed to show the existence of the limit. To accomplish that, we will use the following strategy from \cite{PolstraTuckerCombinedApproach}.

\begin{lem}\label{Lem.LimitExists}\cite[Lemma 3.5.(ii)]{PolstraTuckerCombinedApproach} Let $p$ be a prime number and $d$ be a natural number. Let $a_e$ be a sequence of positive real numbers such that $a_e/p^{e(d+1)}$ is bounded. Suppose that there exists a positive number $C$ such that \begin{equation}\label{Eqn.Existence}\frac{a_e}{p^{e(d+1)}}\leq \frac{a_{e+1}}{p^{(e+1)(d+1)}}+\frac{C}{p^e} \text{ for all }e \in \bN\end{equation} Then the limit $a= \lim_{e\to\infty} \frac{a_e}{p^{e(d+1)}}$ exists.
\end{lem}

 We will check \autoref{Eqn.Existence} for $a_e = \sum_{m=1}^\infty \dim_k \frac{H^0(mL)}{I_e(mL)}$ to show the existence of the $F$-signature. The idea is to find suitable degree-lowering maps and use perturbations to ample divisors to control the $a_e$'s.

\begin{lem}\label{Lem.LimitComparison} Let $L$ be a big divisor on $X$ and fix a constant $C_1 >0$ as in \autoref{Thm:DegenSpaceEqualsWholeSpace}. Suppose there exists a Cartier divisor $D$ such that for each $0 \leq j \leq p-1$, we have an inclusion
\[ \iota_{j}: F_*\cO_X(jL) \hookrightarrow  \cO_X (D)^{\oplus p^{d}}.\]
Then, there exists a constant $C_2>0$ such that for all $e \geq 1$ and for all $m<C_1p^e$, we have \[p^d \dim_k \frac{H^0(D+\lfloor \frac{m}{p}\rfloor L)}{I_e(D+\lfloor\frac{m}{p}\rfloor L)} \leq \dim_k \frac{H^0(mL)}{I_{e+1}(mL)} + C_{2}p^{(e-1)(d-1)}.\]
\end{lem}

\begin{proof}
If $m=0$, we may take $C_2=2p^d\dim_k H^0(X,\cO_X(D))$. Then, \[p^d \dim_k \frac{H^0(D)}{I_e(D)} \leq \dim_k \frac{H^0(X,\cO_X)}{I_{e+1}(0)} + C_{2}p^{(e-1)(d-1)} \text{ for all $e\geq 1$}.\] Hence, we assume $m\geq 1$. For each $m \geq 1$, let us write $m=p\lfloor \frac{m}{p}\rfloor+j_m$ for some $0 \leq j_m \leq p-1$. For each $e\geq 1$, consider an inclusion given by $\iota_{j_m}: F_*\cO_X(j_mL)\to \cO_X(D)^{\oplus p^d}$. Let $\iota_{e,m}$ denote the inclusion
\begin{equation}\label{eqn.IotaMapDefn} \iota_{e,m}: F^{e+1}_* \cO_X(mL) \to F^e_*\cO_X\left(D+\left\lfloor\frac{m}{p} \right\rfloor L\right)^{\oplus p^{d}} \end{equation}
obtained by twisting $\iota_{j_m}$ by the invertible sheaf $\cO_X(\lfloor \frac{m}{p} \rfloor L)$ and taking the $e$-th iterated Frobenius pushforward.

Now, for each $e\geq 1$ and $1\leq m <C_1p^e $, consider the following exact sequence \begin{equation}\label{Eqn.ExactSeqCokercomputation} 0 \to  F^{e+1}_* \cO_X(mL) \stackrel{\iota_{e,m}}{\to} F^e_*\cO_X\left(D+\left\lfloor\frac{m}{p} \right\rfloor L\right)^{\oplus p^{d}} \to \coker(\iota_{e,m}) \to 0. \end{equation}
For simplicity, set $\ell = \lfloor \frac{m}{p} \rfloor $. Then,
\[\coker \iota_{e,m} = \frac{F^e_*\cO_X(D + \ell L)^{\oplus p^d}}{F^{e+1}_*\cO_X((j_m+\ell p)L)} = F^e_*\left(\frac{\cO_X(D)^{\oplus p^d}}{F_*\cO_X(j_mL)}\otimes \cO_X(\ell L) \right).\]

Let $Q_{j}$ be the cokernel of $\iota_j: F_*\cO_X (jL) \to \cO_X(D)^{\oplus p^d}$ and $Z_j$ denote the support of $Q_j$. Since $\iota_j$ is generically an isomorphism, $\dim Z_j \leq d-1$. For each $0 \leq j \leq p-1$, choose a very ample divisor $A_j$ on $Z_j$ such that the support of $A_j - L|_{Z_j}$ is effective and avoids all the associated points of $A_j$. In other words, locally on $Z_j$, the function defining $A_j - L|_{Z_j}$ is a non-zero divisor on $Q_j$. This ensures that for each $j$ and $ \ell \geq 1$, we an injective map
\[ Q_j \otimes \cO_X (\ell L) \hookrightarrow Q_j (\ell A_j). \]

Note that since $\iota_{e,m}$ is obtained by twisting $\iota _{j_m}$, the cokernel of $\iota_{j_m}$ is also supported on $Z_{j_m}$. By the asymptotic Riemann-Roch formula \cite[Example 1.2.9]{LazarsfeldPositivity1}, there exist constants $C', C_0 > 0$ such that \begin{equation}\label{Eqn.CokernelComputation}\begin{split}\dim_k H^0(\coker\iota_{e,m})&= H^0\left(Q_{j_m}\otimes \cO_X(\ell L) \right) \leq H^0\left(Q_{j_m}\otimes \cO_X(\ell A_{j_m}) \right) \\
& \leq   C'\ell^{d-1} +O(l^{d-2})\leq C_0 \ell ^{d-1} \leq  C_0\left(\frac{m}{p}\right)^{d-1}< C p^{(e-1)(d-1)}, \end{split}\end{equation}
where we set $C = C_0C_1^{d-1}$ and we have used that $\ell \leq m/p \leq C_1 p^{e-1}$ since $m \leq C_1 p^e$.

Now, consider a map $\Psi$ defined by the composition of \[H^0(mL) \to H^0\left( D+ \left\lfloor \frac{m}{p}\right\rfloor L\right)^{\oplus p^d}\to \left(\frac{H^0\left( D+ \lfloor \frac{m}{p}\rfloor L\right)}{I_e\left( D+ \lfloor \frac{m}{p}\rfloor L\right)}\right)^{\oplus p^d}.\] We have the exact sequence \[0\to \ker\Psi \to H^0(mL) \stackrel{\Psi}{\to}\left(\frac{H^0\left( D+ \lfloor \frac{m}{p}\rfloor L\right)}{I_e\left( D+ \lfloor \frac{m}{p}\rfloor L\right)}\right)^{\oplus p^d} \to \coker \Psi \to 0.\] Now, we claim that $I_{e+1}(mL)\subseteq \ker\Psi$. Indeed, we have the image of $I_{e+1}(mL)$ in $H^0(D+\lfloor \frac{m}{p}\rfloor L)$ is contained in $ I_e\left( D+ \lfloor \frac{m}{p}\rfloor L\right)$. If not, there is a section $s\in I_{e+1}(mL)$ whose image is not contained in $I_e\left( D+ \lfloor \frac{m}{p}\rfloor L\right)$. Then, via \autoref{eqn.IotaMapDefn}, we have a map $F^{e+1}_*\cO_X(mL) \to F^e_*\cO_X\left( D+ \lfloor \frac{m}{p}\rfloor L\right)^{\oplus p^d}\to \cO_X$ sending $s\mapsto 1$, which contradicts the fact that $s$ is a non-splitting section. 

Therefore,  we have the bounds \[\begin{split} p^d\dim_k\frac{H^0\left( D+ \lfloor \frac{m}{p}\rfloor L\right)}{I_e\left( D+ \lfloor \frac{m}{p}\rfloor L\right)}&\leq \dim_k\frac{H^0(mL)}{\ker \Psi}+\dim_k \coker\Psi \\ &\leq \dim_k\frac{H^0(mL)}{I_{e+1}(mL)}+\dim_k \coker\Psi .\end{split}\]
Now, we bound the dimension of the cokernel of $\Psi$: \begin{align*}\dim_k\coker \Psi &= \dim_k\frac{H^0\left( D+ \lfloor \frac{m}{p}\rfloor L\right)^{\oplus p^d}}{H^0(mL)+ I_e\left( D+ \lfloor \frac{m}{p}\rfloor L\right)^{\oplus p^d}} &\leq \dim_k\frac{H^0\left( D+ \lfloor \frac{m}{p}\rfloor L\right)^{\oplus p^d}}{H^0(mL)}\\
&&  \leq H^0(\coker\iota_{e,m}) \leq Cp^{(e-1)(d-1)} .\end{align*}
The last inequality holds thanks to \autoref{Eqn.CokernelComputation}. This completes the proof.
\end{proof}

\begin{lem} \label{lem:injectivemap}
    Let $X$ be a normal, projective variety of dimension $d$ and $L$ be any Cartier divisor on $X$. Then, there exists a Cartier divisor $D$ such that for each $0 \leq j \leq p-1$, we have an injective map
    \begin{equation} \label{injectivemap} \iota_j: F_* (\cO_X(jL)) \hookrightarrow \cO_X(D) ^{\oplus p^d}.\end{equation}
    Moreover, we may assume that $D$ is effective and defines a normal subscheme of $X$.
\end{lem}
\begin{proof}
    Fix an ample divisor $H$ on $X$. Since $H$ is ample, there exists an $M \gg 0$ such that for all $m \geq M$ and all $0 \leq j \leq p-1$, the sheaf
    \[ \mathscr{H}om_{\cO_X}(F_* (\cO_X(jL)), \cO_X(mH) ) \isom F_* (\cO_X((1-p)K_X - jL )) \otimes \cO_X(mH)    \]
    is generated by its global sections. Here, we have used the duality isomorphism for the Frobenius as expalined in \cite[Section 4.1]{SchwedeSmithLogFanoVsGloballyFRegular}. In particular, we may assume that each of these sheaves is globally generated at the generic point of $X$. Let $\eta$ denote the generic point of $X$ and consider the following restriction map to the generic stalk:    
    \begin{equation} \label{globalgeneratedhom} \Hom_{\cO_X} (F_*(\cO_X(jL)), \cO_X(mH))  \to \Hom_{\cO_{X, \eta}} (F_* (\cO_X(jL)) _{\eta}), \cO_X(mH) _{\eta} ). \end{equation}
    Since $\mathscr{H} om _{\cO_{X}} \big( \,  (F_*(\cO_X(jL)), \cO_X(mH)) \, \big)$ is generically globally generated, this means that the image of the map in \autoref{globalgeneratedhom} generates $\Hom_{\cO_{X, \eta}} (F_* (\cO_X(jL))_{\eta}), \cO_X(mH) _{\eta} )$ as an $\cO_{X, \eta}$-module. Recall that $\cO_{X, \eta}$ is just the fraction field of $X$. Since $F_* (\cO_X(jL)) _{\eta})$ is a free $\cO_{X , \eta}$-vector space of rank $p^{d}$, we can choose $p^{d}$ maps $\phi_{1}, \dots, \phi_{p^{d}}$ in $\Hom_{\cO_X} (F_*(\cO_X(jL)), \cO_X(mH)) $ such that their images under the map in \autoref{globalgeneratedhom} forms a basis of $  \Hom_{\cO_{X, \eta}} (F_* (\cO_X(jL)) _{\eta}), \cO_X(mH) _{\eta} )$ over $\cO_{X, \eta}$. Thus, defining $\iota _{j,m}$ to be the product map 
    \[  \iota _{j,m}: = \phi_{1} \times \dots \times \phi_{p^{d}} : F_*(\cO_X(jL)) \rightarrow \cO_X(mH) ^{\oplus p^{d}} ,\]
    we see that $\iota_{j,m}$ is generically an isomorphism by construction. Furthermore, since $F_* (\cO_X(jL))$ is a torsion-free sheaf of rank $p^{d}$ over $\cO_{X}$, $\iota_{j,m}$ is injective since it is generically an isomorphism. This completes that proof of the claim that maps as in \autoref{injectivemap} exist for each $m \geq M$. Now since $H$ is ample, we may pick $m \gg 0$ so that we can find an effective divisor $D$ in the linear system $|mH|$. Since $X$ is assumed to be normal, by Bertini's theorem, a general such divisor $D$ is  normal by \cite[Corollary 3.4.14]{FlennerJoins}.
\end{proof}

\begin{dfn}
Let $\Delta$ be an effective $\bQ$-Weil divisor on $X$. For any Weil divisor $L$ on $X$ and $e\geq 1$, define the $k$-vector subspace $I^\Delta_e(L)$ of $H^0
(X, \cO_X(L))$ as follows:
\[I_e^\Delta(L) = \{ s \in H^0(L) \mid \varphi(F^e_* s) = 0 \text{ for all } \varphi \in \Hom_{\cO_X}(F^e_*\cO_X(\lceil (p^e-1)\Delta \rceil +L, \cO_X)\}\]
\end{dfn}

Recall that given the effective $\QQ$-divisor $\Delta$, a map $ \varphi \in \Hom_{\cO_X} (F^e_*\cO_X(\lceil (p^e -1) \Delta \rceil + D), \cO_X)$ can be naturally thought of as an element of $ \Hom_{\cO_X} (F^e_*\cO_X(D), \cO_X)$ via the inclusion $\Fe \cO_X (D) \hookrightarrow F^e_*\cO_X(\lceil (p^e -1) \Delta \rceil + D) $. We now recall a result relating various $I_e ^\Delta$ spaces via duality: 

\begin{lem}\cite[Lemma 2.18]{PandeFrobeniusVersionTiansAlphaInvariant}\label{Lem.PandeDuality}
Let $(X,\Delta)$ be a normal projective pair, and let $D$ be any Weil divisor on $X$. Then, denoting $D_1= (1-p^e)K_X-\lceil (p^e-1)\Delta\rceil-D$ and $D_2 = (1-p^e)K_X-\lceil (p^e-1)\Delta\rceil$ for any $e\geq 1$, we have a non-degenerate pairing \[\frac{H^0(D)}{I_e^{\Delta}(D)}\times \frac{H^0(D_1)}{I_e^{\Delta}(D_1)}\to \frac{H^0(D_2)}{I_e^{\Delta}(D_2)}\] obtained from multiplication (and reflexifying) global sections. In particular, \[\dim_k\frac{H^0(D)}{I_e^{\Delta}(D)}=
\dim_k\frac{H^0(D_1)}{I_e^{\Delta}(D_1)}.\]
\end{lem}

Now we estimate the difference between the codimension of $I_e(mL)$ when we perturb  $mL$ by a Cartier divisor. When $L$ is ample, this was done in \cite[Proposition 2.27]{PandeFrobeniusVersionTiansAlphaInvariant} and we treat the case of big divisors here.

\begin{Pn}\label{Pn:twistedFsig}
Let $X$ be a normal projective variety over $k$ of positive dimension and $L$ a big divisor on $X$. Assume $H^0 (X, \cO_X) = k$. Fix a Cartier divisor $D$ on $X$. Then, there exists a constant $C >0$ (depending only on $D$ and $L$) such that
$$ \left| \dim_{k}\frac{H^{0}(mL)}{I_{e}(mL)} - \dim_{k}\frac{H^{0}(mL+D)}{I_{e}(mL+D)} \right| \leq C p^{e(\dim(X) - 1)} $$
for all $m \geq 0$ and $e>0$.

\end{Pn}

\begin{proof}

The case when $m=0$ is straightforward since it is bounded by $1+\dim_kH^0(D)$ which is a constant. Hence, we may assume $m\geq 1$.  First, we consider the case when $D$ is effective. By using the natural map $\cO_{X}(mL) \to \cO_{X}(mL+D)$, we will view $H^{0}(mL)$ as a subspace of $H^{0}(mL+D)$. Let $J_{e}(mL)$ denote the subspace $H^{0}(mL) \cap I_{e}(mL+D)$. By \cite[Lemma 4.12]{LeePandeFsigAmpCone}, we see that $I_{e}(mL) \subset J_{e}(mL)$. Furthermore, recall that we have an inclusion \begin{equation}\label{eqn.IeJeComparison}\frac{H^0 (mL)}{J_e(mL)} \hookrightarrow \frac{H^0(mL+D)}{I_e(mL+D)}\end{equation} coming from the inclusion $H^0(mL)\hookrightarrow H^0(mL+D)$. 
Now, the cokernel of \autoref{eqn.IeJeComparison} is $H^0(mL+D)/(H^0(mL)+I_e(mL+D))$, so this induces
\begin{equation} \label{eqn.DimensionComparison}\dim_{k}\frac{H^{0}(mL+D)}{I_{e}(mL+D)} = \dim_{k}\frac{H^{0}(mL)}{J_{e}(mL)} + \dim_{k}\frac{H^{0}(mL+D)}{H^0 (mL) + I_{e}(mL+D)}.\end{equation}
Using \autoref{eqn.DimensionComparison} and the triangle inequality, we obtain that
\begin{equation} \label{firstestimate} 
\begin{split}
 & \left| \dim_{k}\frac{H^{0}(mL)}{I_{e}(mL)} - \dim_{k}\frac{H^{0}(mL+D)}{I_{e}(mL+D)} \right| \\ 
 = & \left|\dim_k \frac{H^0(mL)}{I_e(mL)}-\dim_k\frac{H^0(mL)}{J_e(mL)} - \dim_k\frac{H^0(mL+D)}{H^0(mL)+I_e(mL+D)}\right|
 \\ \leq &\left|\dim_k \frac{H^0(mL)}{I_e(mL)}-\dim_k\frac{H^0(mL)}{J_e(mL)} \right| + \dim_k\frac{H^0(mL+D)}{H^0(mL)+I_e(mL+D)}
\\  \leq &  \dim_{k}\frac{J_{e}(mL)}{I_{e}(mL)}+\dim_{k} \frac{H^{0}(mL+D)}{H^{0}(mL)} 
\end{split}
\end{equation}
for all $m, e>0$. Now, we bound each term in the above inequality. To compute the first term, we fix an $e>0$ and set $\Delta_e = \frac{1}{p^e -1} D$. Then, we observe that the subspace $J_e (mL)$ is exactly the same as $I_e ^{\Delta_e} (mL)$ by construction. Moreover, we also similarly have 
\begin{equation} \label{dualJe} I_e ^{\Delta_e} ((1-p^e)K_X - mL-D) = H^0 ((1-p^e)K_X - mL -D) \cap I_e ((1-p^e)K_X - mL) .\end{equation}
Thus, by \autoref{Lem.PandeDuality}, we have
\[ \dim_k \frac{H^0 (mL)}{I_e (mL)}  = \dim_k \frac{H^0 ((1-p^e)K_X - mL)}{I_e ((1-p^e)K_X - mL)} \]
and similarly,
\[ \dim_k \frac{H^0 (mL)}{J_e (mL)}  = \dim_k \frac{H^0 ((1-p^e)K_X - mL -D)}{I_e ^{\Delta_e} ((1-p^e)K_X - mL-D)} .\]
By \autoref{dualJe}, we see 
the natural map from $H^0 ((1-p^e)K_X - mL -D)$ to $ H^0 ((1-p^e)K_X - mL)$ restricts to an \emph{injective} map
\[ \frac{H^0 ((1-p^e)K_X - mL -D)}{I_e ^{\Delta_e} ((1-p^e)K_X - mL-D)} \hookrightarrow \frac{H^0 ((1-p^e)K_X - mL)}{I_e ((1-p^e)K_X - mL)} .\]
By considering the cokernel of this map, we get that
\begin{equation} \label{secondestimate}
   \dim_{k}\frac{J_{e}(mL)}{I_{e}(mL)} = \dim_{k}\frac{H^{0}(mL)}{I_{e}(mL)} - \dim_{k}\frac{H^{0}(mL)}{J_{e}(mL)} \leq  \dim_{k} \frac{H^{0}((1-p^e)K_{X} -mL)}{H^{0}((1-p^e)K_{X} -mL-D)}.
\end{equation}

Pick an $M \gg 0$ such that following conditions hold for the linear system $|ML|$: there is an effective divisor $E_2 \in |ML|$ and an effective Weil-divisor $E_3$ such that $E_2 - E_3 \sim -K_X$. We can do this by choosing $M$ such that $H^0 (X,\cO_X(K_X+ ML)) \neq 0$ since $\cO_X(K_X+ML)$ is generically globally generated.


Now, we pick a very ample divisor $H \subset X$ such that there is a divisor $E \in |H|$ that defines a normal subscheme of $X$. This can be done by Bertini's theorems for normal varieties \cite[Corollary 3.4.14]{FlennerJoins}. Moreover, we may do this in such a way that $E$ does not share components with $E_2$ or $E_3$ and there exists an effective Cartier divisor $E_1$ such that $D + E_1 \sim E$. We may also assume, by enlarging $M$ chosen if necessary, that the linear systems $|mL|$ for $m \geq M$ contain a divisor not vanishing along $E$. This is possible since $L$ is big. 
Now, using the divisor $E_1$, we get that
\[  \dim_{k} \frac{H^{0}((1-p^e)K_{X} -mL)}{H^{0}((1-p^e)K_{X} -mL-D)} \leq \dim_k  \frac{H^{0}((1-p^e)K_{X} -mL)}{H^{0}((1-p^e)K_{X} -mL-E)}.  \]
Next, since the divisor $E$ does not share  components with $E_2, E_3 $, for any $e \geq 1$, we may restrict the Weil divisor $(1-p^e) K_X - mL \sim (p^e-1)(E_2 - E_3) - mL$ to $E$ (as explained in \autoref{restrictingWeildivisors}). In conclusion, using \autoref{restrictingtonormaldivisorexactsequence}, we get
\[ \dim_{k}\frac{J_{e} (mL)}{I_{e} (mL)} \leq \dim_{k} H^{0}((p^e-1)(E_2 - E_3)|_E -mL|_E ).  \]
Furthermore, since by assumption, $mL$ also admits a section that does not share a component with $E$, we in fact have:
\begin{equation} \label{thirdestimate}
    \dim_{k} \frac{J_{e} (mL)}{I_{e} (mL)} \leq \dim_k H^0 ((p^e -1) E_2 |_E)  = \frac{\vol(E_2|_E)}{(d-1)!} p^{e(d-1)} + O(p^{e(d-2)})
\end{equation}
 and 
 \begin{equation} \label{fourthestimate}
     \dim_{k} \frac{H^{0}(mL+D)}{H^{0}(mL)}  \leq \frac{\vol(L|_D)}{(d-1)!} m^{d-1} + O(m^{d-2})
 \end{equation}
 
 Finally, by \cite[Theorem 4.9]{LeePandeFsigAmpCone}, we may pick a constant $C_{2}>0$ such that $H^{0}(mL) = I_{e}(mL)$ and $H^{0}(mL+D) = I_{e}(mL+D)$ for $m >C_{2}p^{e}$. Therefore, to prove the Proposition, it is enough to consider the case when $m \leq C_{2} p^{e}$. In this case, the Proposition now follows by putting together inequalities in \ref{firstestimate}, \ref{secondestimate}, \ref{thirdestimate} and \ref{fourthestimate}. This completes the proof of the Proposition in the case when $D$ is effective.

 More generally, when $D$ is not necessarily effective, since $L$ is big, we first pick a fixed number $r \gg 0$ such that $rL$ and $D+ rL$ are both effective. Then, for any $ e \geq 1$ and $m > r$, we have
\begin{equation*}
\begin{split}
 \left| \dim_{k}\frac{H^{0}(mL)}{I_{e} (mL)} - \dim_{k}\frac{H^{0}(mL+D)}{I_{e}(mL+D)} \right| \leq & \left| \dim_{k}\frac{H^{0}(mL)}{I_{e} (mL)} - \dim_{k}\frac{H^{0}((m-r)L)}{I_{e}  ((m-r)L)} \right| \\
 + &\left| \dim_{k}\frac{H^{0}((m-r)L)}{I_{e}  ((m-r)L)} - \dim_{k}\frac{H^{0}((m-r)L + rL+D)}{I_{e}  ((m-r)L+rL+D)} \right| .
 \end{split}   \end{equation*}
 Now, we may apply the previous case of the Proposition (since both $D+rL$ and $rL$ are effective) to each of the two terms in the above inequality. Since $r$ was independent of $e$, this completes the proof of the Proposition.
\end{proof}

\begin{proof}[Proof of \autoref{thm:Fsigbigexists}]
Set \[a_e = \sum_{m=0}^\infty \dim_k\frac{H^0(mL)}{I_e(mL)}.\] To show the existence of $\s_X(L)$, using \autoref{Lem.LimitExists}, we need to check the following: \begin{enumerate}
\item $\displaystyle\frac{a_e}{p^{e(d+1)}}$ is bounded above.
\item there is some constant $C$ such that $\displaystyle\frac{a_e}{p^{e(d+1)}}\leq \frac{a_{e+1}}{p^{(e+1)(d+1)}}+\frac{C}{p^e} \text{ for all }e \in \bN.$
\end{enumerate}
We first prove part (a). Using \autoref{Thm:DegenSpaceEqualsWholeSpace}, we fix $C_0>0$ such that $H^0(mL) = I_e(mL)$ whenever $m> C_0p^e$. Then, \[ a_e = \sum_{m=0}^{ C_0p^e} \dim_k\frac{H^0(mL)}{I_e(mL)} \leq \sum_{m=0}^{C_0p^e}\dim_k H^0(mL) \leq \tilde{C}\vol(L)p^{e(d+1)} + O(p^{ed}) \]
for some $\tilde{C} >0$, where the last inequality uses \autoref{lem:sumvol}.
Dividing by $p^{e(d+1)}$, we see that $a_e/p^{e(d+1)}$ is bounded above. 

For part (b), choose a Cartier divisor $D$ and a injective map $\iota_j : F_*\cO_X(jL)\to \cO_X(D)^{\oplus p^d}$ for each $0 \leq j \leq p -1$, as in \autoref{lem:injectivemap}.
We again fix a constant $C_0>0$ such that $H^0(mL) = I_e(mL)$ for all $m>C_0p^e$. Then \[\begin{split} p^d\dim_k \frac{H^0(\lfloor \frac{m}{p}\rfloor L)}{I_e(\lfloor \frac{m}{p} \rfloor L)} & = p^d \left(\dim_k \frac{H^0(\lfloor \frac{m}{p}\rfloor L)}{I_e(\lfloor \frac{m}{p} \rfloor L)} - \dim_k \frac{H^0(\lfloor \frac{m}{p}\rfloor L + D)}{I_e(\lfloor \frac{m}{p}\rfloor L + D)} \right) + p^d\dim_k \frac{H^0(\lfloor \frac{m}{p}\rfloor L + D)}{I_e(\lfloor \frac{m}{p}\rfloor L + D)} \\ & \leq p^d C_1p^{e(d-1)} +  p^d\dim_k \frac{H^0(\lfloor \frac{m}{p}\rfloor L + D)}{I_e(\lfloor \frac{m}{p}\rfloor L + D)} \text{  (by \autoref{Pn:twistedFsig})} \\ &\leq C_1 p^{e(d-1)+d}+\dim_k \frac{H^0(mL)}{I_{e+1}(mL)} + C_{2}p^{(e-1)(d-1)} \text{ (by \autoref{Lem.LimitComparison})}.\end{split}\] 
Dividing by $p^{(e+1)(d+1)}$ for both sides, we obtain \begin{equation}\label{eqn.Fsigexistscomputation}\frac{1}{p^{e(d+1) + 1}}\dim_k \frac{H^0(\lfloor \frac{m}{p}\rfloor L)}{I_e(\lfloor \frac{m}{p} \rfloor L)} \leq \frac{1}{p^{(e+1)(d+1)}}\dim_k \frac{H^0(mL)}{I_{e+1}(mL)} +
\frac{C_1}{p^{2e+1}} + O(\frac{1}{p^{2e + 2d +2}} ).\end{equation} 
Moreover, we pick a constant $C_3 > 0$ such that the error term $\frac{C_1}{p^{2e+1}} + O(\frac{1}{p^{2e + 2d +2}} ) \leq \frac{C_3}{p^{2e+1}}$ for all $m$.

 We now sum both sides of \autoref{eqn.Fsigexistscomputation} over all $m$. Note that the terms of this sum are non-zero only for $m \leq p\lceil C_0 p^{e} \rceil : = pM$. Thus, the left-hand side of the sum may be written as \[\frac{1}{p^{e(d+1) + 1}}\sum_{m=0}^{pM}\dim_k \frac{H^0(\lfloor \frac{m}{p}\rfloor L)}{I_e(\lfloor \frac{m}{p} \rfloor L)} = \frac{p}{p^{e(d+1) + 1}} \sum_{m=0}^{M}\dim_k \frac{H^0(mL)}{I_e(mL)} = \frac{a_e}{p^{e(d+1)}}\]
 On the other hand, summing up the right-hand side of \autoref{eqn.Fsigexistscomputation} over $m = 0, \dots, Mp$, we have
\[\begin{split}  \frac{1}{p^{(e+1)(d+1)}}\sum_{m=0}^{Mp}\dim_k \frac{H^0(mL)}{I_{e+1}(mL)} +\sum_{m=0}^{Mp}
\frac{C_3}{p^{2e+1}} & =\frac{a_{e+1}}{p^{(e+1)(d+1)}} +
\frac{C_3(Mp+1)}{p^{2e+1}} \\  \leq \frac{a_{e+1}}{p^{(e+1)(d+1)}}+ \frac{C_3 (C_0 p^{e+1} + p)}{p^{2e+1}} & \leq \frac{a_{e+1}}{p^{(e+1)(d+1)}}+ \frac{2C_0C_3}{p^{e}}. \end{split}\]
Hence, we have shown that for all $e \geq 1$, we have \[\frac{a_e}{p^{e(d+1)}}\leq \frac{a_{e+1}}{p^{(e+1)(d+1)}} + \frac{C}{p^e}\]
where $C = 2 C_0 C_3$. Applying \autoref{Lem.LimitExists} completes the proof that the limit defining $\s_X(L)$ exists.
\end{proof}

\begin{Pn}[Transformation rule of $F$-signature for big divisors]\label{Pn:TransformationRuleFsig}
    Let $X$ be a normal projective variety over $k$ and $L$ be a big Cartier divisor on $X$. Then, for any integer $n \geq 1$, we have the transformation rule
    \[ \s_X(L^n) = \frac{\s_X(L)}{n}.\]
\end{Pn}
\begin{proof}
The case $n=1$ is immediate, so we may assume $n\geq 2$. To show the identity, we will rearrange the sums as follows. \[\sum_{m=0}^\infty \dim_k\frac{H^0(mL)}{I_e(mL)} = \sum_{i=0}^{n-1}\sum_{q=0}^{\infty}\dim_k\frac{H^0(nqL + iL)}{I_e(nqL+iL)}\]
By \autoref{Pn:twistedFsig}, for each $i = 1,2, \dots, n-1$, there exists a constant $C_{i}>0$ such that \[\left|\dim_k\frac{H^0(nqL+iL)}{I_e(nqL+iL)}-\dim _k\frac{H^0(nqL)}{I_e(nqL)}\right| < C_{i} p^{e(\dim (X) -1)}\] for all $e>0, m\geq0$. Let $\tilde{C} = \max{C_i}$. Furthermore, by \autoref{Thm:DegenSpaceEqualsWholeSpace}, there is a constant $C$ such that $H^0(nqL) = I_e(nqL)$ and $H^0(nqL+iL) = I_e(nqL+iL)$ for all $q>Cp^e/n$ and $i=0,\dots, n-1$. Hence,
\[\begin{split} &\left|\sum_{m=0}^{\infty} \dim_k\frac{H^0(mL)}{I_e(mL)}-n\sum_{m=0}^{\infty} \dim_k\frac{H^0(mnL)}{I_e(mnL)}\right|\\
=&\left|\left(\sum_{i=0}^{n-1}\sum_{q=0}^{\infty} \dim_k\frac{H^0(nqL+iL)}{I_e(nqL+iL)}-\dim_k\frac{H^0(nqL)}{I_e(nqL)}\right)\right|\\
\leq & \sum_{i=0}^{n-1}\sum_{q=0}^{\lfloor Cp^e/n\rfloor} \left|\dim_k\frac{H^0(nqL+iL)}{I_e(nqL+iL)}-\dim_k\frac{H^0(nqL)}{I_e(nqL)}\right| \\
\leq & n \left\lfloor\frac{Cp^e}{n}\right\rfloor\tilde{C}p^{e(\dim(X) -1)} \leq Mp^{e\dim(X)} \text{ for some constant $M>0$.} 
\end{split}\]
Dividing both sides by $p^{e(\dim(X)+1)}$ and taking $e\to\infty$ we have \[\s_X(L) -n\s_X(L^{n}) = \lim_{e\to\infty}\frac{1}{p^{e(\dim(X)+1)}}\left(\sum_{m=0}^{\infty} \dim_k\frac{H^0(mL)}{I_e(mL)}-n\sum_{m=0}^{\infty} \dim_k\frac{H^0(mnL)}{I_e(mnL)}\right) = 0.\]
\end{proof}

\autoref{Pn:TransformationRuleFsig} allows us to define the $F$-signature of a big $\Q$-Cartier $\Q$-divisor.
\begin{dfn} \label{dfn:FsigofbigQdiv}
Let $X$ be a normal projective variety over $k$ and $L$ be a big $\bQ$-Cartier $\Q$-divisor. Then, we define \[\s_X(L):=r\s_X(L^r)\]
for any $r > 0$ such that $rL$ is a Cartier divisor. This is well-defined by \autoref{Pn:TransformationRuleFsig}. 
\end{dfn}

\begin{rem}
    We will prove in \autoref{thm:positivityofFsigbig} that a normal projective variety $X$ is globally $F$-regular if and only if the $F$-signature $\s_X (L)$ is positive for \emph{some} (equivalently, any) big divisor $L$. 
\end{rem}

\begin{rem}
If $X$ is globally $F$-regular, then there are no numerically trivial $\ZZ$ or $\QQ$-divisors over $X$. In other words, numerical equivalence on $X$ is the same as linear equivalence \cite[Theorem 3.4]{LeePandeFsigAmpCone}. Hence, we regard the $F$-signature function $\s_X$ as a well-defined function on the rational big cone of $X$.
\end{rem}

\begin{Pn}\label{Prop:signaturevolumecomparison}
Let $X$ be a $d$-dimensional normal globally $F$-regular variety over $k$. Fix a norm $\| \cdot \|$ on the N\'eron-Severi space of $X$. Then, there exists a constant $C$ such that for any big $\bQ$-divisor $L$, \[ \s_X(L) \leq \frac{C\vol(L)}{\|L\|^{d+1}}.\] 
\end{Pn}

\begin{proof}
It suffices to show the existence of such $C$ for integral big divisors due to the transformation rule. By \autoref{Thm:DegenSpaceEqualsWholeSpace}, there exists $C_1>0$ such that for any big divisor $L$, $I_e(mL) = H^0(mL)$ for all $m>\frac{Cp^e}{\|L\|}.$ Hence, we may write 
\[\begin{split} \s_X(L) & = \lim_{e\to\infty}\frac{1}{p^{e(d+1)}}\sum_{m=0}^\infty \dim_k\frac{H^0(mL)}{I_e(mL)}\\ & = \lim_{e\to\infty} \frac{1}{p^{e(d+1)}}\sum_{m=0}^{\lfloor C_1p^e/\|L\|\rfloor}\dim_k\frac{H^0(mL)}{I_e(mL)} \\ &\leq \lim_{e\to\infty}\frac{1}{p^{e(d+1)}}\sum_{m=0}^{\lfloor C_1p^e/\|L\|\rfloor}\dim_k H^0(mL) \\ &\leq \lim_{N\to\infty} \frac{C_1^{d+1}}{N^{d+1}\|L\|^{d+1}}\sum_{m=0}^N\dim_k H^0(mL) \\ & = \frac{C_1^{d+1}}{(d+1)!\|L\|^{d+1}}\vol(L) \text{ (by \autoref{lem:sumvol}).}\end{split}\]
Therefore, the theorem follows.
\end{proof}

\subsection{Continuity of $F$-signature function on the big cone} So far, we have defined the $F$-signature function on the cone of big $\Q$-divisors over $X$, denoted $\text{Big} _\QQ (X)$. In this subsection, we study the continuity properties of this function which will allow us to extend the $F$-signature to all big $\RR$-divisors. This is a generalization of \cite{LeePandeFsigAmpCone} where we studied the $F$-signature function on the ample cone. Throughout, we fix a norm $\| \cdot \|$ on the N\'eron-Severi space of $X$.

\begin{thm}\label{thm:FsigLocallyLipschitz}
Let $X$ be a globally $F$-regular projective variety over $k$. There is a well-defined and continuous $F$-signature function on the big cone of $X$
\[ \s_X: \text{Big} (X) \to \mathbb{R}_{\geq 0}\]
extending \autoref{dfn:FsigofbigQdiv} on big $\QQ$-divisors and the $F$-signature function from \cite{LeePandeFsigAmpCone} on the ample cone. Moreover, $\s_X$ is locally Lipschitz continuous extension on the big cone satisfies the identity $\s_X( tL) = \frac{\s_X (L)}{t}$ for all $t >0$ and $L \in \text{Big} (X)$. Furthermore, for any closed cone $\mathcal{C}\subseteq N^1(X)_\bR$ with $\mathcal{C}\setminus\{0\}\subseteq \mathrm{Big}(X)$, there exists $M$ such that 
\[|\s_X(L)-\s_X(L')|\leq \frac{M}{\min\{\|L\|,\|L'\|\}^2}\|L-L'\|\] for all nonzero $L,L'\in\mathcal{C}$.
In particular, for any $\epsilon>0$, the $F$-signature function is Lipschitz continuous on $\{L\in\mathcal{C}\mid \|L\|\geq \epsilon\}$.
\end{thm}

\begin{Not}
We denote $C\subseteq \mathbb{R}^k$ as \emph{a pointed convex cone} if it satisfies 
\begin{itemize} 
\item $\mathrm{int}(C)$ is not empty
\item for each $x\in C$ and $\lambda \in \mathbb{R}_{>0}$, $\lambda x\in C$
\item for each $x,y\in C$, $x+y\in C$
\item $\bar{C}\cap (-\bar{C}) = \{0\}$ where $\bar{C}$ is the closure of $C$.
\end{itemize} Furthermore, we denote $C_\bQ := C\cap \mathbb{Q}^k$ be rational points in $C$.
\end{Not}

 We begin with the following useful lemma, which is a convex-geometry variant of \cite[Theorem 4.2]{LeePandeFsigAmpCone}.

\begin{lem}\label{lem:LipschitzContinuousExtension}
Let $C\subset \bR^k$ be an open pointed convex cone and $\| \|$ be a fixed norm on $\bR^k$. Let $f:C_\bQ\to \bR$ be a function defined on nonzero rational points on $C$ such that for each real $z\in C$, there exist constants $A(z)>0$ and $R(z)>0$ such that if $x,y\in C_\bQ\cap B_{R(z)}(z)$ with $x-y\in C$, then \[|f(x)-f(y)| < A(z)\|x-y\|.\] Then, $f$ is locally Lipschitz continuous on $C$ (i.e., even when $x-y$ is not contained in $\mathrm{int}(C)$). As a consequence $f$ admits a unique continuous extension $\tilde f:C\rightarrow \bR$ and this extension is locally Lipschitz.  
\end{lem}

\begin{proof}
Let $z\in C$. We aim to shrink $R(z)$ so that $f$ is locally Lipschitz in every direction. Fix $z_0\in C_\bQ$ a rational point in the cone. Consider $\epsilon >0$ such that $B_\epsilon(z_0)\subseteq C$. Let $K=1+\frac{4\|z_0\|}{\epsilon}$ and $r(z) := \frac{R(z)}{2K}$. Now, we take arbitrary two rational points $x,y\in B_{r(z)}(z)\cap C_\bQ$. We choose a rational number $t$ such that $\frac{\|x-y\|}{\epsilon} < t < \frac{2\|x-y\|}{\epsilon}.$ Now, set $w=x+tz_0$ a rational point. Here, $w\in B_{R(z)}(z)$ since \[\begin{split}\|w-z\| = \|x+tz_0 -z\| &\leq \|x-z\| + t\|z_0\|\\
&\leq \|x-z\| + \frac{2\|x-y\|}{\epsilon}\|z_0\|\\
&\leq r(z)+\frac{4r(z)}{\epsilon}\|z_0\| \\ 
& = r(z) K = \frac{R(z)}{2}.
\end{split}\] Furthermore, we can observe that both $w-x$ and $w-y$ are in $C_\bQ$. In particular, we see $w-y = t((x-y)/t + z_0)$ and $(x-y)/t + z_0 \in B_{\epsilon}(z_0)\subseteq C$. Hence, by the assumption, 
\[\begin{split}|f(x)-f(y)| & \leq |f(x)-f(w)|+|f(w)-f(y)|\\
& \leq A(z)\|x-w\| + A(z) \|w-y\| \\ 
& \leq A(z) \left( t\|z_0\| + \|x-y\| + t\|z_0\|\right) \\
& \leq A(z) \left(\frac{4\|x-y\|}{\epsilon}\|z_0\| + \|x-y\|\right) \\ 
& = A(z)K\|x-y\|.
\end{split}\] Therefore, the claim follows.
\end{proof}

\begin{Pn}\label{Pn:LocallyLipschitzwithTransformationRule}
Assume the hypotheses of \autoref{lem:LipschitzContinuousExtension}. Furthermore, assume $f(tx)=\frac{1}{t}f(x)$ for any $x\in C_\bQ$ and any positive rational number $t$. Then, its continuous extension also satisfy $\tilde{f}(tx) = \frac{1}{t}\tilde{f}(x)$ for any $x\in C$ and $t>0$. Furthermore, let $\mathcal{C}\subseteq \bR^k$ be a closed cone such that $\mathcal{C}\setminus{\{0\}}\subseteq C$. Then, there exist $R>0$ and $M>0$ such that for every $z\in \mathcal{C}\setminus{\{0\}}$, $B_{R\|z\|}(z) \subseteq C$ and \[|\tilde{f}(x)-\tilde{f}(y)|\leq \frac{M}{\|z\|^2}\|x-y\|\] for any $x,y\in B_{R\|z\|}(z)$.
\end{Pn}

\begin{proof}
The continuous extension admits the transformation since $\tilde{f}$ is continuous. Now, consider a compact subset \[\kappa=\{\xi \in \mathcal{C}\mid 1 \leq \| \xi \| \leq 2\}.\] By \autoref{lem:LipschitzContinuousExtension}, there exists $K>0$ such that for any $\xi \in \kappa$, we can find $r(\xi)$  such that $\tilde{f}$ is Lipschitz continuous on $B_{r(\xi)}(\xi)$ with a Lipschitz constant $KA(\xi)$. Here, since $\cup B_{r(\xi)/2}(\xi)$ covers $\kappa$ and $\kappa$ is compact, we can find a finite subcover with centers $\xi_1,\dots ,\xi_N$. Let $R= \frac{1}{2}\min r(\xi_i)$ and $M = K\max A(\xi_i)$. Now, take any $z\in \mathcal{C}\setminus\{0\}.$ Let $w = z/\|z\|$ be the normalization of $z$. Since $w\in \kappa$, $w$ is in $B_{r(\xi_i)/2}(\xi_i)$ for some $i$. Therefore, we have $B_R(w)\subseteq B_{r(\xi_i)}(\xi_i)$. With this, consider arbitrary two elements $x,y\in B_{R\|z\|}(z)$. Then, \[\begin{split}|\tilde{f}(x)-\tilde{f}(y)|&=\frac{1}{\|z\|}\left|\tilde{f}\left(\frac{x}{\|z\|}\right)-\tilde{f}\left(\frac{y}{\|z\|}\right)\right|\\
& \leq \frac{1}{\|z\|}M\left\|\frac{x}{\|z\|} - \frac{y}{\|z\|}\right\| \\
& = \frac{M}{\|z\|^2}\|x-y\|.
\end{split}
\] Therefore, the theorem follows.
\end{proof}

\begin{thm}\label{Thm:LipschitzAwayFromZero}
With the hypothesis as \autoref{Pn:LocallyLipschitzwithTransformationRule}, there is a constant $\tilde{M}>0$ such that for any $x,y\in \mathcal{C}\setminus\{0\}$, \[|\tilde{f}(x)-\tilde{f}(y)| \leq \frac{\tilde{M}}{\min\{\|x\|,\|y\|\}^2}\|x-y\|.\] As a consequence $\tilde{f}$ is Lipschitz continuous on $\{x\in \mathcal{C}\mid \|x\|\geq \epsilon \}$ for any $\epsilon>0$.  
\end{thm}
\begin{proof}
Let $L=\{\xi \in \mathcal{C} \mid \|\xi \| =1\}$ be a compact subset of $C$. Let $B = \max_{u\in L}|\tilde{f}(u)|$. Here, due to the transformation rule, we may observe $|\tilde{f}(x)| = \frac{|\tilde{f}(x/\|x\|)|}{\|x\|}\leq B/\|x\|$ for any $x\in \mathcal{C}\setminus\{0\}.$ Now, to prove the theorem, consider two arbitrary elements $x,y\in \mathcal{C}\setminus\{0\}$. Without loss of generality, assume $\|x\|\leq \|y\|$. Now, suppose that $\|x-y\| <R\|x\|$. Then, by \autoref{Pn:LocallyLipschitzwithTransformationRule}, if we set $z = x$, we have \[ |\tilde{f}(x)-\tilde{f}(y)|\leq \frac{M}{\|x\|^2}\|x-y\|=\frac{M}{\min\{\|x\|,\|y\|\}^2}\|x-y\|\] We note here that this $M$ is independent of the choice of $x$ and $y$. On the other hand, if $\|x-y\|\geq R\|x\|$, we have \[|\tilde{f}(x)-\tilde{f}(y)|\leq \frac{B}{\|x\|}+\frac{B}{\|y\|} \leq \frac{2B}{\min\{\|x\|,\|y\|\}}\leq \frac{2B}{R\min\{\|x\|,\|y\|\}^2}\|x-y\|.\] Here, $B$ and $R$ are also independent to $x$ and $y$. Hence taking $\tilde{M}=\max\{M,2B/R\}$ proves the theorem.
\end{proof}

We recall the following estimate on the difference of the $F$-signature under small, effective perturbations:

\begin{lem}\label{lem.BoundsForFsigBigDiv}(cf. \cite{LeePandeFsigAmpCone})
Let $X$ be a $d$-dimensional normal variety over $k$. Let $L$ be a big Cartier divisor and $H$ be an effective Cartier divisor on $X$. Let $n$ and $b:=b(n)$ (depending only on $n$) be positive integers such that $n>2\|H\|/\|L\|$ and that $nL-bH$ is big. Then, there exists a constant $C$ depending only on $X$ such that \[\left|\s_X(L) - \s_X\left(L+\frac{1}{n}H\right)\right|\leq \frac{C}{\|L\|^{d+1}}\left(2\vol(L)\frac{(b+1)^d - b^d}{b^d} + \vol\left(L+\frac{1}{n}H\right)-\vol(L)\right) + 2\frac{\s_X(L)}{b+1}.\] 
\end{lem}

\begin{proof}
By \autoref{MainthmA}, we know the $F$-signature exists. Thus, the inequality follows from the dimension computations of the non-splitting subspaces of the global sections from the proof of \cite[Lemma 4.14]{LeePandeFsigAmpCone}. We omit the computations here.
\end{proof}

\begin{proof}[Proof of \autoref{thm:FsigLocallyLipschitz}]
By \autoref{Pn:TransformationRuleFsig}, the $F$-signature function has the transformation rule $\s_X(L^n) = \frac{1}{n}\s_X(L)$ for any $n\in \mathbb{N}$ and a big divisor $L$. The transformation rule naturally extends to positive rational numbers. To show the $F$-signature function is locally Lipschitz continuous on the big cone, it remains to show that the difference is bounded when we move in an effective direction. 

Let $D\in \mathrm{Big}(X)$ be a real big divisor class. Since the big cone is open, there exists a constant $r>0$ such that $\overline{B_{2r}(D)}\subseteq \mathrm{Big}(X)$ and $r<\frac{\|D\|}{2}$. Consider big $\bQ$-divisors $L$ and $L'$ in $B_{r/4}(D)$ such that $L'-L$ is big. Then, we may observe \[ \|L'-L\| \leq \|L'-D\|+\|L-D\| \leq \frac{r}{2}.\]
Furthermore, we also have \begin{equation}\label{eqn:lowerboundforL} \|L\| \geq \|D\|-\|L-D\|\geq 2r-\frac{r}{4}=\frac{7}{4}r.\end{equation}
Hence, $L$ and $L'$ are big $\bQ$-divisors such that $\|L'-L\| < \frac{\|L\|}{2}$ and that $L'-L$ is big. Now, let us write $L'-L = \frac{1}{n}H$ for some large $n\in \bN$ where $H$ is an effective divisor. Let $b(n) := \lfloor \frac{nr}{\|H\|}\rfloor$ so that $b(n)\to \infty$ as $n\to \infty$. Furthermore, by construction, $2\frac{\|H\|}{\|L\|} = n\frac{2\|L'-L\|}{\|L\|} < n$ and $nL-bH$ is big. Indeed, we observe $L-\frac{b}{n}H\in B_{2r}(D)$ since \[\left\|D-\left(L-\frac{b}{n}H\right) \right\|\leq \|D-L\| + \frac{b}{n}\|H\|\leq \frac{r}{4}+r < 2r.\] 
Now, by \autoref{lem.BoundsForFsigBigDiv}, 
\[\begin{split}\left|\s_X(L) - \s_X\left(L+\frac{1}{n}H\right)\right|&\leq \frac{C}{\|L\|^{d+1}}\left(2\vol(L)\frac{(b+1)^d - b^d}{b^d} + \vol\left(L+\frac{1}{n}H\right)-\vol(L)\right) + 2\frac{\s_X(L)}{b+1}\\ 
 &\leq \frac{C}{\|L\|^{d+1}}\left(2\vol(L)\frac{2^d}{b} + \vol\left(L+\frac{1}{n}H\right)-\vol(L)\right) + \frac{2}{b}\s_X(L)\\
\end{split}
\] Now, note that for any big $\bQ$-divisors $L_1$ and $L_2$, \[|\vol(L_1) - \vol(L_2)|\leq C_0 \max (\|L_1\|, \|L_2\|)^{d-1} \|L_1 - L_2\|\] for some constant $C_0$ depending only on $X$ and the norm $\| \|$. Here, recall that we picked $L,L'\in B_{r/4}(D)$ where $r<\frac{\|D\|}{4}$. Hence, both $\|L\|$ and $\|L'\|$ are less than $2\|D\|$. Further, $b(n) := \lfloor \frac{nr}{\|H\|}\rfloor\geq \frac{nr}{2\|H\|}$. If we substitute $H = n(L'-L)$ and all information we had above, the inequality becomes 
\[\begin{split}|\s_X(L)-\s_X(L')| & \leq \frac{C}{\|L\|^{d+1}}\left(\frac{2^{d+2}}{r}\vol(L)\|L-L'\| + 2^{d-1}C_0 \|D\|^{d-1} \|L-L'\|\right) + \frac{4}{r}\s_X(L)\|L-L'\|\\
& \leq \left(\frac{C}{\|L\|^{d+1}}\left(\frac{2^{d+2}}{r}\vol(L)+2^{d-1}\|D\|^{d-1}C_0\right) + \frac{4}{r}\s_X(L)\right) \|L-L'\|.
\end{split}\] 
Now, \autoref{Prop:signaturevolumecomparison}, 
$\s_X(L)\leq C_1\frac{\vol(L)}{\|L\|^{d+1}}$ for some uniform constant $C_1$. Furthermore, the volume function is continuous on the big cone, so we let $M=\max_{L\in \overline{B_{2r}(D)}}\vol(L)$. By \autoref{eqn:lowerboundforL}, we have $\|L\|\geq \frac{7}{4}r$. Hence, \[|\s_X(L)-\s_X(L')| \leq \left(C\left(\frac{4}{7r}\right)^{d+1}\left(\frac{2^{d+2}M}{r}+2^{d-1}\|D\|^{d-1}C_0\right)+7C_1M\left(\frac{4}{7r}\right)^{d+2}\right)\|L-L'\|.\]
Here, the Lipschitz constant on the above inequality depends only on $\|D\|$. Indeed $C$ and $C_0$ are independent to the divisors we chose. $r$ and $M$ are constants depend only $D$. 

Therefore, by \autoref{lem:LipschitzContinuousExtension}, the $F$-signature function is locally Lipschitz continuous on the big cone and admits the continuous extension to the big cone. By \autoref{Thm:LipschitzAwayFromZero}, there exists $M$ such that \[|\s_X(L)-\s_X(L')|\leq \frac{M}{\min\{\|L\|,\|L'\|\}^2}\|L-L'\|.\]
As a consequence, we conclude that the $F$-signature function is Lipschitz continuous on $\{L\in \mathcal{C} \mid \|x\| \geq \epsilon\}$ where $\epsilon >0$ is a positive number and  $\mathcal{C}$ is a closed cone such that $\mathcal{C}\setminus\{0\}\subset \mathrm{Big}(X)$.
\end{proof}

\section{The $\FA$-invariant on the big cone.}\label{Section:AlphaInvariant}
Throughout this section, let $X$ be a normal, projective variety over $k$ with $\mathrm{dim}(X)>0$.

\subsection{The $\FA$-invariant.} \label{section:alphainvarinat}

\begin{dfn} \label{alphae}
 Let $L$ be a big Cartier divisor on $X$. For each integer $e \geq 1$, we define
\[ m_{e} (X, L) : = \max \left(\{ m \geq 0 \, | \, I_{e} (mL)  = 0 \}\cup\{0\}\right), \]
where $I_e$ is the subspace defined in \autoref{Iedfn}. We also define
\[ \alpha_{e}(X,L) := \frac{m_e (X,L)}{p^e} .\]
\end{dfn}
Note that since $|mL| \neq \emptyset$ for all $m \gg 0$, we must have $m_e (X,L) < \infty$ for all $e$, see \autoref{Thm:DegenSpaceEqualsWholeSpace}.
\begin{lem}\label{lem:AlphaInvExists}
     The limit $\lim_{e \to \infty} \alpha_e (X,L)$ exists.
\end{lem}
\begin{proof}
We first claim that for all $e \geq 1$, we have
     \begin{equation}\label{monotonicity} \alpha_{e}  + \frac{1}{p^e} \geq \alpha_{e+1}  + \frac{1}{p^{e+1}}. \end{equation}
     Let $0  \neq f $ be an element of $I_e (m_e + 1)$ (which exists by definition). Then, we have that $f^{p} $ is a non-zero element of $I_{e+1} (p(m_e + 1))$ (see \cite[Lemma 3.11]{PandeFrobeniusVersionTiansAlphaInvariant}). In other words, $I_{e+1}(p(m_e+1))\neq 0$. Hence, it implies
    \[ p(m_e +1) \geq m_{e+1} + 1. \]
    Dividing both sides by $p^{e+1}$, we obtain the inequality \autoref{monotonicity}.
   Therefore, the sequence $\{ \alpha_e + \frac{1}{p^e} \}_{e \geq 1}$ is decreasing sequence of non-negative real numbers. Hence, the sequence converges to its infimum. Moreover, since the sequence $\frac{1}{p^{e}}$ converges to zero, the sequence $\{ \alpha_e \}_{e \geq 1}$ also converges and
   \[ \lim _{e \to \infty} \alpha_{e} = \inf _{e \geq 1} \{ \alpha_e + \frac{1}{p^e} \}. \]
\end{proof}

\begin{dfn} \label{Falphadfn}
    We define the $\FA$-invariant of $(X,L)$ as
    \[ \FA(X,L) := \lim_{e \to \infty} \alpha_e (X,L).\]
\end{dfn}

\begin{rem}
    When $L$ is an ample divisor, this definition of $\FA$-invariant can be interpreted in terms of $F$-pure thresholds and is an analog of Tian's $\alpha$-invariant in positive characteristics \cite{PandeFrobeniusVersionTiansAlphaInvariant}.
\end{rem}

\begin{Pn}\cite{PandeFrobeniusVersionTiansAlphaInvariant} \label{Prop:transformationforalpha}
Let $X$ be a normal projective variety, and $L$ be a big Cartier divisor. Then, for any $n\geq 1$, \[\alpha_F(X,nL) = \frac{1}{n}\alpha_F(X,L)\]
\end{Pn}

\begin{proof}
Fix $M\gg0$ such that $| mL| \not= 0$ for all $m\geq M$. We claim $m_e(L) \leq nm_e(nL)\leq m_e(L) + M$. For the first inequality, it suffices to show $ \lfloor m_e(L)/n \rfloor \leq m_e(nL)$ since $m_e$'s are integers. This follows from the definition since $n\lfloor m_e(L) / n\rfloor \leq m_e(L)$, which implies $I_e(n\lfloor m_e(L)/n\rfloor L) = 0$. Now, we prove $nm_e(nL) \leq m_e(L)+ M$. We may assume $nm_e(nL) - M >0$. If not, we have $nm_e(nL) - M \leq m_e(L)$ that already gives the required inequality.

Set $k = nm_e(L) - M$. Since $M\leq nm_e(nL) - k$, we have $|(nm_e(nL) - k)L|\not=\emptyset$. Hence, we can take an effective divisor $E\in |(nm_e(nL) - k)L|$ such that $kL + E\sim nm_e(nL)L$. By construction, we have $I_e(kL+E) = 0$. Now, since $I_e(kL) \subseteq I_e(kL+E)$, we have $I_e(kL) = 0$, or equivalently, $k\leq m_e(L)$. Therefore, $nm_e(L) - M \leq m_e(L)$. Now, dividing the inequality by $1/p^e$ and take the limit $e\to \infty$, we have $\alpha_F(X, nL) = \frac{1}{n}\alpha_F(X,L)$.
\end{proof}

With the transformation rule above, the Frobenius alpha invariant extends to big $\bQ$-divisors.
\begin{dfn} \label{dfn:alphaofQQdiv}
    Let $X$ be a normal projective variety over $k$ and $L$ be a big, $\QQ$-Cartier, $\QQ$-divisor. Then, we define $\FA(X,L) $ by the formula
    \[ \FA(X,L) = n \, \FA(X,nL)  \]
    where $n \geq 1$ is any integer such that $nL$ is a Cartier divisor. Thanks to \autoref{Prop:transformationforalpha}, the definition of $\FA(X,L)$ is independent of the choice of $n$.
\end{dfn}

In the rest of this section, we will establish the key properties of the $\FA$-invariant on the big cone.

\subsection{Monotonicity and Positivity}

\begin{Pn} \label{Prop:addingeffective}
    Let $X$ be a normal projective variety and $L$ be a big $\mathbb{Q}$-Cartier divisor on $X$. Let $E$ be any effective $\mathbb{Q}$-Cartier divisor on $X$. Then we have
 \[ \FA(X,L) \geq \FA(X, L+E) .\]
\end{Pn}
\begin{proof}
    Let $n$ and $m$ be the Cartier index of $L$ and $E$ respectively. Let $L'= mnL$ and $E'=mnE$. Using \autoref{lem:Ieeffective}, we see that for any $e \geq 1$, we have $m_e (X,L') \geq m_e (X, L'+E')$. Now, the proposition follows from \autoref{Prop:transformationforalpha} and the definition of $\FA$-invariant.
\end{proof}

\begin{Pn} \label{Prop:Fsigandalphacomparison}
    Let $X$ be a normal, projective variety and $L$ be a big Cartier divisor. Assume that $d = \dim(X)$ is positive. Fix a constant $C>0$ such that for all $e \geq 1$ and all $m \geq C p^e$, we have $I_e (mL) = H^0 (mL)$. Then, we have the inequality
    \[ \frac{\vol(L) \, \FA(X, L)^{d+1}}{(d+1)!} \leq  \s_X(L) \leq \frac{\vol(L)}{(d+1)!} \, \big( C^{d+1} - (C - \FA(X,L) )^{d+1} \big) . \]
\end{Pn}
\begin{proof}
    The first inequality follows immediately from the definition of the $F$-signature of $L$, once we use the fact that for any $e \geq 1$, $I_e (m) = 0$ for $m \leq m_{e}(X,L)$ (see \autoref{alphae}). For the second inequality, we use \autoref{postivitylemma} proved below along with the fact that for any real number $\lambda > \FA(X,L)$, there exists an $e \gg 1 $ and a section $f \in I_e (nL)$ such that $\lambda > \frac{n}{p^e -1 }$. The inequality then follows by taking a limit $\lambda \to \FA(X,L)$.
\end{proof}

\begin{lem} \cite[Lemma 4.7] {PandeFrobeniusVersionTiansAlphaInvariant} \label{postivitylemma}
Fix a constant $C>0$ such that for all $e \geq 1$ and all $m \geq C p^e$, we have $I_e (mL) = 0$. Suppose we have integers $e_0,n \geq 1$ and a non-zero section $f \in I_{e_0} (nL) \subset H^0 (X, \cO_X (nL))$ and set $\lambda \geq  \frac{n}{p^{e_0} - 1} $ be any real number. Then,
  \begin{equation}
        \s_X(L) \leq  \frac{\vol(L)}{(d+1)!} \, \big( C^{d+1} - (C - \lambda) ^{d+1} \big) .
    \end{equation}
\end{lem}
\begin{proof}
Since $f $ belongs to $I_{e_{0}} (nL)$, for each $ m \geq n$ we know that $f \cdot H^ 0 (X, \cO_X ((m -n) L)) \subset H^0 (X, \cO_X (mL))$ is contained in $ I_{e_{0}}(mL)$ yielding the inequality
\begin{equation}
  \dim_{k} \frac{H^ 0 (X, \cO_X (mL))}{I_{e_{0}}(mL)} \leq  h^ 0 (X, \cO_X (mL)) - h^ 0 (X, \cO_X ((m- n)L))  .  
\end{equation}
for all $m \geq n$.
Further, setting $v_{r} = \frac{p^{re_{0}} }{p^{e_0} - 1} $ for any positive integer $r$, we claim that $f^{v_r} $ belongs to  $I_{re_{0}}(n v_{r} L )$. To see this, let $D$ denote the effective divisor corresponding to $f$ and note that we have a natural map of $\cO_X$-modules $\cO_X \to F_* ^{e_0} (\cO_X (D)) $. Twisting this by $D$ and applying $F_* ^{e_0}$, we obtain a composition:
\[ \cO_X \to F_* ^{e_0} \cO_X (D) \to F_* ^{2e_0} \cO_X(( 1+ p^{e_0}) D) \]
where the composite map is the same as the composition $\cO_X \to F_* ^{2e_0} \to F_* ^{2e_0} \cO_X ((p^{e_0} + 1) D)$. Proceeding inductively, we obtain a factorization
\[ \cO_X \to F_* ^{e_0} \cO_X ( D) \to F_* ^{re_0} \cO_X (v_r D). \]
Therefore, any splitting of the map $\cO_X \to F_* ^{re_0} \cO_X (v_r D)$ as provides a splitting of $ \cO_X \to F_* ^{e_0} \cO_X ( D) $. This shows that $f^{v_r} \in I_{re_0} (nv_r  L)$.

Therefore, we have
\begin{equation} \label{upperboundinequality}
  \dim_{k} \frac{H^ 0 (X, \cO_X (mL))}{I_{ne_{0}}(m)} \leq h^0 (X, \cO_X(mL)) - h^0 (X, \cO_X((m-nv_{r}) L) )  
\end{equation}
for all $m \geq n v_{r}$. Then, using the \autoref{Dfn.FsigBigDivisors} we may bound the $F$-signature $\s_X(L)$ as follows:
$$ \s_X(L) =  \lim _{r \to \infty } \frac{  \sum \limits _{m = 0} ^{ Cp^{re_{0}}} \dim_{k} \frac{H^0 (mL)}{I_{re_{0}}(mL)}}{p^{re_{0}(d+1)}} \leq \Bigg( \lim _{r \to \infty } \frac{  \sum \limits _{m = 0} ^{ Cp^{re_{0}}} h^0 (\cO_X(mL))}{p^{re_{0}(d+1)}} - \frac{  \sum \limits _{m = nv_{r}} ^{ Cp^{re_{0}}} h^0 (\cO_X((m- n v_{r}) L) )}{p^{re_{0} (d+1)}} \Bigg).$$
Finally, calculating the dimensions in the above inequality using \autoref{lem:sumvol},
we obtain
\[    \s_X(L) \leq \frac{\vol(L)}{(d+1)!} \big( C^{d+1} - (C - \frac{n}{p^{e_{0}} -1}) ^{d+1} \big).  \]
The proof of the lemma is now complete by using the fact that $\lambda \geq \frac{n}{p^{e_0} -1}$.
\end{proof}

As a consequence, we deduce the following characterization of global $F$-regularity in terms of the $F$-signature of \emph{any} big divisor over $X$.

\begin{thm} \label{thm:positivityofFsigbig}
    Suppose $X$ is a normal projective variety and $L$ is a big divisor over $X$. Suppose $X$ is globally $F$-regular. Then, both $\FA(X,L)$ and $\s_X (L)$ (\autoref{dfn:Fsigbig}) are positive.
    
   Conversely, suppose that the $F$-signature $\s_X (L)$ as defined in \autoref{Dfn.FsigBigDivisors} is positive. Then, $X$ is globally $F$-regular.
\end{thm}

\begin{proof}
    First assume that $X$ is globally $F$-regular. Then, by \cite[Theorem 4.6]{PandeFrobeniusVersionTiansAlphaInvariant}, the invariant $\FA(X, H)$ is positive for any ample divisor $H$. Pick a very ample divisor $H$ such that $H - L = E$ is linearly equivalent to an effective divisor $E$. Then by \autoref{Prop:addingeffective}, we have $\FA(X, L) \geq \FA(X, H = L +E) > 0 .$ Therefore, we see that $\FA (X, L) >0$. Using \autoref{Prop:Fsigandalphacomparison}, we see that $\s_X (L) > 0$ as well since $\vol (L) >0$ for the big divisor $L$.

    For the converse, we will show that if the $F$-signature of a big divisor $L$ is positive, then the $F$-signature of any ample divisor $H$ over $X$ is positive as well, from which the theorem follows by the local case proved in \cite{AberbachLeuschke}. First, we claim that if $\s_ X(L)$ is positive, then $\FA(X,L)$ is positive as well. This follows immediately from \autoref{Prop:Fsigandalphacomparison}. Let $H$ be any ample Cartier divisor on $X$. Since $L$ is assumed to be big, we have $nL - H$ is effective for some $n \gg 0$. Therefore, we may write $nL = H + E$. By the monotonicity (\autoref{Prop:addingeffective}) and the transformation property (\autoref{Prop:transformationforalpha}) of the $\FA$-invariant, we obtain
    \[ \FA(X,H) \geq \FA(X, nL) = \frac{1}{n} \FA(X,L) > 0 .\]
    Therefore, we have $\FA(X,H) >0 $ and hence by \autoref{Prop:Fsigandalphacomparison} again, we have $\s_X (H) > 0$.
\end{proof}

\subsection{Continuity of the $\FA$-invariant on the big cone}

Now we show that the $\alpha_F$-invariant is continuous on the big cone. This is significantly easier than the case of the $F$-signature thanks to its monotonicity.

\begin{thm} \label{thm:alphaonbigcone}
    Let $X$ be a globally $F$-regular projective variety over $k$. Then, the $\FA$-invariant defines a well-defined function $\FA(X, \_)$ on the real big cone of $X$
    \[\FA(X, \_) : \text{Big}_\RR(X) \to \RR\]
    satisfying the following properties:
    \begin{enumerate}
        \item $\FA(X,L)$ is a locally Lipshcitz function of $L$.
        \item For any big rational class $L$, $\FA(X,L)$ is the same as defined in \autoref{dfn:alphaofQQdiv}.
        \item For any big real class $L$ and any integer $n \geq 1$, we have $\FA(X,nL) = \frac{\FA(X,L)}{n}$.
        \item  For any big real class $L$ and any effective class $E$ in $\text{N}^1 _\RR (X)$, we have $\FA(X,L) \geq \FA(X, L+E)$.
        \item For any big real class $L$, we have $\FA(X,L)$ is positive.
    \end{enumerate}
\end{thm}

\begin{lem}\label{lem.AlphaInvInequality}
Suppose that there is a number $0 <a<1$ such that $aL+D$ and $aL-D$ are both big. Then,  \[|\alpha_F(X, L) - \alpha_F(X,L+D)| \leq \frac{a}{1-a}\alpha_F(X,L)\]
\end{lem}

\begin{proof}
Applying \autoref{Prop:addingeffective}, we have \begin{equation}\label{Eqn.ContinuityofAlphaInv1}\begin{split}\alpha_F(X, L) &= (1+a)\alpha_F(X,(1+a)L)  \\ &= (1+a)\alpha_F(X, (L+D) + (aL-D)) \\ &\leq (1+a) \alpha_F(X,L+D) \end{split}\end{equation} In other words, we have $\alpha_F(X,L) - \alpha_F(X,L+D)\leq a\alpha_F(X,L+D)$.
On the other hand, \begin{equation}\label{Eqn.ContinuityofAlphaInv2}\begin{split}\alpha_F(X, L) &= (1-a)\alpha_F(X,(1-a)L)  \\ &= (1-a)\alpha_F(X, (L+D) - (aL+D)) \\ &\geq (1-a) \alpha_F(X,L+D) \end{split}\end{equation} Combining \autoref{Eqn.ContinuityofAlphaInv1} and \autoref{Eqn.ContinuityofAlphaInv2}, we obtain $|\alpha_F(X, L) - \alpha_F(X,L+D)| \leq a\alpha_F(X,L+D) \leq \frac{a}{1-a}\alpha_F(X,L)$. \end{proof}

\begin{lem} \label{lem:upperboundonalpha}
Let $X$ be a normal projective variety over $k$. Fix a norm $\| \, \|$ on the N\'eron-Severi space of $X$. Then there exists a constant $C>0$ (depending only on $X$) such that for all big $\QQ$-divisors $L$, we have
\[ \FA(X,L) \leq \frac{C}{\|L\| } .\]
\end{lem}
\begin{proof}
    By the homogeneity of $\FA(X,L)$ (\autoref{Prop:transformationforalpha}), it is sufficient to prove the lemma for all big Cartier divisors $L$. From \cite[Theorem~4.9]{LeePandeFsigAmpCone}, we have a constant $C>0$ such that for all $e>0 $ and all big Cartier divisors $L$, we have $I_e (mL) = H^0 (mL)$ for all $m \geq \frac{Cp^e}{\| L \|}$. In particular, for a fixed big divisor $L$ and for all $e \gg 0$ so that $H^0 (mL) \neq 0$ for $m \geq \frac{Cp^e}{\| L \|}$, we have $m _e (L) \leq \frac{Cp^e}{\| L \|} $. It follows that $\FA(X, L) \leq \frac{C}{\|L\|}$ for all big $\QQ$-divisors $L$.
\end{proof}
\begin{lem}\label{lem:Alphainvariantlipschitzeffectivedirection}
Let $X$ be a normal projective variety and $L$ be a big $\RR$ divisor. Fix a norm $\| \, \|$ on the N\'eron-Severi space of $X$. Let $L$ be a big $\mathbb{R}$-divisor on $X$ Then, there are constants $C, \delta >0$ depending only on $L$ such that for any two big $\QQ$-divisors $L_1, L_2\in B_\delta(L)$ such that $L_2-L_1$ is effective and $\|L_2-L_1\|<\frac{\|L_2\|}{2}$, we have \[|\alpha_F(X, L_1) - \alpha_F(X,L_2)|\leq C \| L_1 - L_2 \|.\] 

\end{lem}

\begin{proof}
 Fix $0<r <\frac{1}{2}\|L\|$ such that $B_r(L) \subseteq \mathrm{Big}(X)$. Let $\delta = \frac{r}{4}$. Let $L_1, L_2\in B_\delta(L)$ be big $\QQ$-divisors and $L_2-L_1$ is effective. Let $D= L_2 - L_1$. Note that $L_1+D$ and $L_1-D$ are in $B_{\frac{1}{2}r}(L)$, hence they are both big. Now, we observe that $L_1+\left(1+\frac{r}{4\|D\|}\right)D$ and $L_1-\left(1+\frac{r}{4\|D\|}\right)D$ are in $B_r(L)$ since\[\begin{split}\left\|L_1+\left(1+\frac{r}{4\|D\|}\right)D - L\right\| &\leq \left\|D+\frac{r}{4\|D\|}D\right\| + \|L_1 - L\|\\ & \leq \| D \| +\left\|\frac{r}{4\|D\|}\right\|\|D\|+\|L_1 - L\|\\ &< \frac{1}{2}r + \frac{1}{4}r + \frac{1}{4}r =r \end{split}\] Let $a = \frac{4\|D\|}{r+4\|D\|}$. Then, we have $aL_1-D$ and $aL_1+D$ are big and $0<a<1$. Thus, by \autoref{lem.AlphaInvInequality}, we have \[\begin{split}|\alpha_F(X,L_2) - \alpha_F(X,L_1)| &\leq \frac{a}{1-a}\alpha_F(X,L_1)\\
& = \frac{4}{r}\alpha_F(X,L_1)\|L_2-L_1\|\\ & \leq \frac{4C}{r\|L_1\|}\|L_2-L_1\| \text{ (by \autoref{lem:upperboundonalpha})} \\
& \leq \frac{32C}{7r\|L\|}\|L_2-L_1\|.
\end{split}\] The last inequality is because we had $\|L_1-L\| <\frac{1}{4}r$ and $r<\frac{1}{2}\|L\|$. Hence, $\|L\| = \|L-L_1+L_1\| \leq \|L-L_1 \| + \|L_1\| \leq \frac{1}{8}\|L\| + \|L_1\|$. Thus, we have $\frac{7}{8}\|L\| \leq \|L_1\|$.
\end{proof}

\begin{Pn} \label{Prop:alphacontinuity}
Let $X$ be a normal projective variety and $L$ be a big $\RR$ divisor. Fix a norm $\| \, \|$ on the N\'eron-Severi space of $X$. Then, there are constants $C, \delta >0$ such that for any two big $\QQ$-divisors $L_1, L_2$ such that $\|L_i - L\| < \delta$ for $i=1,2$, we have \[|\alpha_F(X, L_1) - \alpha_F(X,L_2)|\leq C \| L_1 - L_2 \|.\] 
In other words, the function $\FA(X , .)$ is locally Lipschitz around every real big class $L$ in the N\'eron-Severi space.
\end{Pn}

\begin{proof}
We know that the Frobenius $\alpha$ invariant satisfies the transformation rule by \autoref{Prop:transformationforalpha}. Furthermore, by \autoref{lem:Alphainvariantlipschitzeffectivedirection}, it is locally Lipschitz continuous in effective directions. By \autoref{lem:LipschitzContinuousExtension}, we conclude that the alpha invariant is locally Lipschitz in any direction. 
\end{proof}

\begin{proof}[Proof of \autoref{thm:alphaonbigcone}]
For any big $\RR$-divisor $L$ on $X$, we may define $\FA(X,L)$ by the formula
\[ \FA(X,L) = \lim_{m \to \infty} \FA(X,L_m)  \]
for any sequence of big $\QQ$-divisors $L_m$ converging to $L$ in $\text{N}^1 _\RR (X)$. By \autoref{Prop:alphacontinuity}, this is well-defined and makes $\FA(X,L)$ a locally Lipschitz continuous function of $L$. This proves Parts (a) and (b). Part (c) follows from the continuity of $\FA(X,L)$ and \autoref{Prop:transformationforalpha}. Similarly, writing both $L$ and $E$ as a limit of big $\QQ$-divisors $L_m$ and $E_m$, Part (d) follows from \autoref{Prop:addingeffective}.

For Part (e), we first pick a big $\QQ$-divisor $L'$ such that $L' - L$ is effective. This can be done by the openness of the big cone. By Part (d), we have that $\FA(X,L) \geq \FA(X,L')$. Therefore, it is sufficient to show that $\FA(X.L')$ is positive. Moreover, by the transformation rule Part (c), we may further assume that $L'$ is a big Cartier divisor. Now pick an ample Cartier divisor $H$ and write $L'+ E  = nH$ for some $n \gg 0$ and an effective Cartier divisor $E$. This can be done since $H$ is ample. Therefore, we have
 \[ \FA(X,L') \geq \FA(X, L'+ E) = \FA (nH) > 0 \]
where in the last inequality we are using the fact that the $\FA$-invariant of an ample divisor on a globally $F$-regular projective variety is positive \cite[Theorem~4.6]{PandeFrobeniusVersionTiansAlphaInvariant}. This completes the proof of the theorem.
\end{proof}

\subsection{Positivity of the $F$-signature for big $\RR$-divisors.}
\begin{thm} \label{thm:positivityforbig}
Let $X$ be a normal projective variety over $k$ of positive dimension $d$. Assume that $X$ is globally $F$-regular. Then for any big $\RR$-divisor $L$ on $X$, the $F$-signature $\s_X (L)$ is positive. Moreover, for any compact subset $K$ that is contained in the big cone, there exists a constant $C>0$ (depending on $K$) such that $\s_X(L) \geq C \vol(L)$ for all $L  \in K$.
\end{thm}
\begin{proof}
    First we note that by using \autoref{Prop:Fsigandalphacomparison} and by continuous extension, for any big $\RR$-divisor $L$ we have
    \begin{equation} \label{eqn:Fsigalphainequality} \s_X(L) \geq \frac{\vol(L) \, \FA(X, L)^{d+1}}{(d+1)!}.  \end{equation}
    Now, since the function $\FA(X,L)$ is positive on the big cone (by Part (e) of \autoref{thm:alphaonbigcone}), the function $\FA(X, .)$ achieves a positive minimum on any compact subset contained in the big cone. Choose $M$ to be the minimum value of $\FA(X,.)$ on $K$. Then, choosing $C = \frac{M ^{d+1}}{(d+1)!}$ works for the Theorem by \autoref{eqn:Fsigalphainequality}.  The positivity of $\s_X (L)$ now follows from the positivity of $\vol(L)$ since $L$ is assumed to be big.
\end{proof}

\begin{rem}
    A special case of the above theorem is that the $F$-signature of ample $\RR$-divisors is positive on a globally $F$-regular variety. Even this case of the theorem is new and require the techniques involving the $\FA$-invariant developed in this section.
\end{rem}

\section{Global $F$-regularity and the $F$-signature function for graded sequences.} \label{section:gradedsequences}

In this section, we introduce the notion of globally $F$-regular pairs $(X, \fad)$, where $X$ is a projective variety over $k$ and $\fad$ is a graded sequence of coherent ideal sheaves on $X$. We also consider their $F$-signatures via the formalism of Cartier subalgebras as used in \cite{BlickleSchwedeTuckerFSigPairs1}.

\subsection{Global $F$-regularity and the $F$-signature.}

\begin{dfn} \label{gradedsequencedfn}
    Let $X$ be a normal variety over $k$. Then a \emph{graded sequence of ideal sheaves} on $X$, denoted by $\fad$, is a sequence of coherent ideal sheaves $(\fa_n)_{n \geq 1}$ such that for any integers $n,m \geq 1$, we have
    \[ \fa_m \cdot \fa_n \subset \fa_{m+n}.  \]
    In other words, for every affine open set $U \subset X$, the sequence $(\fa_n (U))_{n \geq 1}$ forms a graded sequence of ideals of $\cO_X(U)$. We will follow the convention that $\fa_0 = \cO_X$ for any graded sequence of ideal sheaves $\fad$. 
\end{dfn}

\begin{dfn} \label{def:gFrforgradedsequence}
Let $X$ be a normal variety over $k$ and $\fad$ be a graded sequence of ideal sheaves on $X$ (\autoref{gradedsequencedfn}). We say that the pair $(X, \fad)$ is \emph{globally $F$-regular} if for any effective Weil divisor $D$, there is an $e > 0$, and an element
\[\varphi\in H^0(\Fe\frak{a}_{p^e-1} \cdot \sheafhom_{\cO_X}(F^e_*\cO_X(D),\cO_X))\] 
such that $\varphi(F^e_* 1)=1$, where $\Fe 1$ is the global section of $\Fe
\cO_X (D)$ that corresponds to the effective divisor $D$. Note that via the natural inclusion
\[ F^e_*\frak{a}_{p^e-1} \cdot \sheafhom_{\cO_X}(F^e_*\cO_X(D),\cO_X)) \subset  \sheafhom_{\cO_X}(F^e_*\cO_X(D),\cO_X), \]
we may consider any element $\varphi \in H^0(\Fe\frak{a}_{p^e-1} \cdot \sheafhom_{\cO_X}(F^e_*\cO_X(D),\cO_X))$, as a map of $\cO_X$-modules from $\Fe \cO_X(D)$ to $\cO_X$.
\end{dfn}

It will be convenient to think of the splitting condition in the above definition in terms of the subspaces $I_e ^{\fad} $ as below:
\begin{dfn} \label{def:Iebullet}
    Let $(X,\fad)$ be a pair consisting of a normal projective variety $X$ and a graded sequence of ideal sheaves $\fad$. For any Weil divisor $D$ on $X$, we define the subspace $I_e^{\fad}(D)$ as \[I_e^{\fad}(D):= \{f\in H^0(\cO_X (D)) \mid \varphi(F^e_*f) = 0 \text{ for all } \varphi \in H^0(F^e_*\frak{a}_{p^e-1} \cdot \sheafhom_{\cO_X}(F^e_*\cO_X(D),\cO_X)). \}\]
\end{dfn}

Note that in the above definition, the subspace $I_e ^{\fad} (D)$ only depends on the sheaf $\cO_X(D)$ and not on the divisor $D$ chosen. In this notation, the pair $(X, \fad)$ is globally $F$-regular if and only if for any effective Weil divisor $D$, the global section $1 \in H^0 (X, \cO_X(D))$ is not contained in $I_e ^{\fad} $ for some $e >0$. Let $E$ be an effective divisor. Throughout this section, we will write $E\in I^{\fa_\bullet}_e(D)$ for any effective divisor $E \sim D$ if the global section of $\cO_X (D)$ defining $E$ belongs to $I^{\fa_\bullet}_e(D)$. 

With this notation, we can define the $F$-signature of a pair $(X, \fad)$ where $X$ is projective, with respect to an ample line bundle $\cL$ over $X$:
\begin{dfn} \label{def:fsigofidealsequence}
Let $(X, \fa_\bullet)$ be a pair consisting of a normal projective variety of dimension $d$ and a graded sequence of ideal sheaves $\fad$. We define the $F$-signature function of $\s_X(\fa_\bullet, L)$ for any ample divisor $L$ on $X$ as
\[ \s_X(\fad,L) = \lim_{e \to\infty} \frac{1}{p^{e(d+1)}}\sum_{r=0}^\infty \dim \frac{H^0(rL)}{ I^{\fad}_e(rL)}.\]
\end{dfn}

The existence of the limit in the above definition requires a proof, which is one of the goals of the rest of this section. We will prove:

\begin{thm} \label{thm:fsigofgradedsequences}
Let $X$ be a normal variety over $k$ and $\fa_\bullet$ be a graded sequence of ideal sheaves on $X$. Let $L$ be an ample divisor on $X$. Then, the $F$-signature $\s_X(\fa_\bullet, L)$ exists and is positive if and only if $(X,\fa_\bullet)$ is globally $F$-regular (\autoref{def:gFrforgradedsequence}).
\end{thm}

We will prove this theorem by reducing it to the case of a \emph{Cartier subalgebra} and importing results from \cite{BlickleSchwedeTuckerFSigPairs1}. We begin by showing that in \autoref{def:gFrforgradedsequence}, we may restrict to checking splittings along certain ample divisors. The following operations on maps will be useful:

\begin{dfn} \label{def:globaloperations}
    Given a normal variety $X$ over $k$, we define the following operations on $p^{-e}$-linear maps on $X$:
    \begin{itemize}
       
         \item Given two ideal coherent sheaves $\fa$ and $\fb$, and Weil divisors $D_1$ and $D_2$, and global sections $\varphi$ and $\psi$ of
         $ F^e_*\frak{a} \cdot \sheafhom_{\cO_X}(F^e_*\cO_X(D_1),\cO_X)$ and $ F^f_*\frak{b} \cdot \sheafhom_{\cO_X}(F^f_*\cO_X(D_2),\cO_X)$ respectively, we have the map $\varphi \star \psi \in  F^{e+f} _* (\fa ^{[p^f]} \fb) \cdot( \sheafhom_{\cO_X}(F^{e+f}_*\cO_X(D_2 + p^f D_1),\cO_X) )$ defined as follows: on any affine open subset $U$ of $X$, we pick local sections $g \in \fa$ and $h \in \fb$ such that $\varphi = \tilde{\varphi} ( \Fe g \rule{0.8em}{0.2pt})$ and  $\psi = \tilde{\psi} ( F^f _ * h \rule{0.8em}{0.2pt})$ for maps $\tilde{\varphi} \in \Hom_{\cO_U} ( \Fe \cO_U (D_1) , \cO_U)$ and $\tilde{\psi} \in \Hom_{\cO_U} (F_* ^f \cO_U (D_2) , \cO_U)$, and we define 
         \[ \varphi\star \psi (U) := \tilde{\varphi} \circ F^e_* \big (g\cO_U ( D_1) \otimes_{\cO_U}  \psi  \big )^{**}. \]
        Here $M^{**}$ denotes the reflexification of the module $M$. More concretely, we may first twist the map $\psi$ by the rank one reflexive module $g \cO_U (D_1)$ and reflexify to get a map  
        \[  \big ( g\cO_U ( D_1) \otimes_{\cO_U}   \psi \big )^{**} :  F_* ^{f} \big ( g^{p^f} \,h \, \cO_U (D_2 + p^f D_1) \big ) \to g\cO_U (D_1) . \]
       Then, we may pushforward the above map by $\Fe$ and post-compose with $\varphi$, which is the resulting map
       \[  \varphi \star \psi (U) : F_* ^{e+f} \big (g^{p^f} \, h \, \cO_U (D_2 + p^f D_1) \big) \to \Fe (g\cO_U (D_1)) \to \cO_X. \]
       This construction naturally globalizes to define a global section $\varphi \star \psi $ of the sheaf  $F^{e+f} _* (\fa ^{[p^f]} \fb) \cdot( \sheafhom_{\cO_X}(F^{e+f}_*\cO_X(D_2 + p^f D_1),\cO_X) )$.
       \medskip
       
          \item Given a coherent ideal sheaf $\fa \subset \cO_X$ and a Weil divisor $D$ and a global section $\varphi$ of the sheaf
         $ F^e_*\frak{a} \cdot \sheafhom_{\cO_X}(F^e_*\cO_X(D),\cO_X)$ for some $e \geq 1$, we define $\varphi^{[p^f]}$ as the iteratively as $\varphi ^{[p^{f+1}]} = \varphi \star \varphi^{[p^{f}]} $ (for any integer $f \geq 1$). We note that $\varphi ^{[p^f]}$ is naturally a global section of the sheaf $ F^{ef} _* (\fa ^{1 + p^e + \dots + p^{(f-1)e}}) \cdot( \sheafhom_{\cO_X}(F^e_*\cO_X( (1 + p^e + \dots + p^{(f-1)e}) D),\cO_X) ) $
    \end{itemize}
\end{dfn}

\begin{Pn} \label{prop:testdivisor}
    A pair $(X, \fad)$ as in \autoref{def:gFrforgradedsequence} with $X$ \emph{projective} is globally $F$-regular if and only if one of the following conditions hold:
    \begin{enumerate}
        \item there exists an ample divisor $L$ on $X$ such that for any $r \geq 0$ and any effective divisor $D \sim r L $, there is an $e > 0$ such that $D \notin I_e ^{\fad} (rL)$. 

    \item there exists an ample and effective divisor $D_0 \geq 0$ such that:
    \begin{itemize}
        \item  the vanishing locus $ V(\fa_1)$ of $\fa_1$ is contained in the support of $D_0$,

        \item the variety $X \setminus \Supp (D_0) $ has strongly $F$-regular singularities and

        \item there is an $e_0 > 0$ such that $D_0 \notin I_{e_0} ^{\fad} (\cO_X(D_0))$. 
    \end{itemize}
    \end{enumerate}
\end{Pn}

\begin{lem} \label{lem:testdivisor}
    Let $D$ be a Weil divisor on a normal projective variety $X$ and let $f \in H^0 (X, \cO_X (D))$ be a global section. We have the following properties:
    \begin{enumerate}
       \item If $D'$ is another Weil divisor and $g \in H^0 (X, \cO_X (D'))$ such that $fg \notin I_e ^{\fad} (\cO_X(D + D'))$, then $f \notin I_e ^{\fad} (\cO_X (D))$.

        \item If $f \notin I_{e_0} ^{\fad} (\cO_X(D))$, then $ f^{1 + p + \dots + p^{e-1}}$ is not contained in $I_{e e_0} ^{\fad} ((1 + p + \dots + p^{e-1})D)$ for all $e \geq 1$.

    \end{enumerate}
\end{lem}

\begin{proof} Recall the $\star$ and Frobenius-power operations on maps introduced in \autoref{def:globaloperations}.

     \begin{enumerate}
         \item Let $D_f$ and $D_g$ be the effective divisors defined by $f$ and $g$. Note that we have a natural inclusion $\cO_X (D_f ) \subset \cO_X(D_f + D_g)$ (mapping the section $f$ to $fg$) inducing an inclusion
        \[  \sheafhom_{\cO_X}(F^e_*\cO_X(D_f + D_g),\cO_X)) \subset \sheafhom_{\cO_X}(F^e_*\cO_X(D_f),\cO_X)). \]
        By pre-multiplication by sections of $\Fe \fa_{p^e -1}$, we then have an inclusion
         \[  \Fe \fa_{p^e -1 } \cdot \sheafhom_{\cO_X}(F^e_*\cO_X(D_f + D_g),\cO_X)) \subset \Fe \fa_{p^e -1} \cdot \sheafhom_{\cO_X}(F^e_*\cO_X(D_f),\cO_X)). \]
        Then, we see that any element $\psi \in H^0 (\Fe \fa_{p^e -1 } \cdot \sheafhom_{\cO_X}(F^e_*\cO_X(D_f + D_g),\cO_X)) $ such that $\psi (\Fe 1) = 1$ gives an element $\psi \in H^0 (\Fe \fa_{p^e -1} \cdot \sheafhom_{\cO_X}(F^e_*\cO_X(D_f),\cO_X))$ as required.

        \item Let $D_f $ denote the effective Weil divisor corresponding to $f$ (so that $D_f \sim D$). Since $f \notin I_{e_0} ^{\fad} (\cO_X (D))$, there exists a section $\varphi \in H^0(F^{e_0} _* \fa_{p^{e_0 - 1}} \cdot \sheafhom_{\cO_X}(F _ *  ^{e_0} \cO_X(D_f),\cO_X)) $ that maps the global section $F_* ^{e_0} 1$ to $1$. Then, it follows from the construction of $\varphi^{[p^e]}$ that it defines a map
        \[ \varphi^{[p^e]}: F^{ee_0}_*\cO_X( (1 + p^{e_0} + \dots + p^{(e-1)e_0} ) D_f),\cO_X)   \]
        such that $ \varphi^{[p^e]} (F_* ^{e e_0} 1)  = 1$ (see \cite[Lemma 5.2.3]{BlickleSchwedeSurveyPMinusE} for a similar argument). Moreover, since $\varphi$ is premultiplied by $\fa_{p^{e_0} -1}$, we have $\varphi ^{[p^e]}  $ is premultiplied by sections of $\fa_{p^{e_0} - 1} ^{\frac{p^{ e e_0  -1}}{p^{e_0} -1}} \subset \fa_{p^{e e_0} - 1}$. This proves part (b) of the lemma. \qedhere
    \end{enumerate}
\end{proof}

\begin{proof}[Proof of \autoref{prop:testdivisor}]
The forward direction for both parts follows from the observation that if $(X, \fad)$ is globally $F$-regular, then we know already that $X$ is globally $F$-regular. So we will now prove the converse.
Part (a) of the Proposition follows immediately from Part (a) of \autoref{lem:testdivisor} since for any effective divisor $D$, we may find an $r>0$ and an effective divisor $E$ such that $D + E \sim rL $. 

Now we prove Part (b): We note that since $D_0$ is assumed to be ample, the complement $U = X \setminus \Supp (D_0)$ is an affine variety with strongly $F$-regular singularities. Now, for any effective divisor $D \sim r D_0$ (for some $r >0$), we have a splitting $ \phi: \Fe \cO_U (D|_U) \to \cO_U$ for some $e \gg 0$. We will view $\phi$ as a section of the coherent sheaf $\sheafhom_{\cO_U}(F^e_*\cO_U(D|_U),\cO_U)$ over the open set $U$ and let $t_0$ denote the section corresponding to $D_0$ of $\cO_X (D_0)$. Since $V(\fa_{p^e - 1})   \subset V (\fa_ 1 ^{p^e -1}) =V(\fa_1) \subset \text{Supp}(D_0) $, we have a restriction map of sections:
\[ H^0 (X,    F^e_*\frak{a}_{p^e -1} \cdot \sheafhom_{\cO_X}(F^e_*\cO_X(D),\cO_X) ) \to H^0 (U, \sheafhom_{\cO_U}(F^e_*\cO_U(D|_U),\cO_U)) . \]
Since $U = X \setminus \text{Supp} (D_0)$ and $D_0$ is ample, the section $\phi t_0^n$  extends to $X$ along the above map for some $n \gg 0$ to a global section 
\[ \tilde{\phi}  \in H^0 (X,    F^e_*\frak{a}_{p^e -1} \cdot \sheafhom_{\cO_X}(F^e_*\cO_X(D),\cO_X (nD_0) ) \]
such that the restriction of $\tilde{\phi}$ to $U$ is equal to $\phi$. In particular, the map $\tilde{\phi}$ sends the canonical section $\Fe 1 $ to $1 \in \cO_X (nD_0)$ (defined by the effective divisor $nD_0$). Next, by \autoref{lem:testdivisor} we may choose $e' > 0$ and a map $\psi \in H^0 (X, F_ * ^{e'}  \fa _{p^{e'} - 1} \cdot \sheafhom (F_* ^{e'} \cO_X  (nD_0), \cO_X)$ such that $\psi (F_* ^{e'} 1) =  1$. Then, noting that $\tilde{\phi} (-nD_0)$ can be viewed as a map in $   H^0 (X,    F^e_*\frak{a}_{p^e -1} \cdot \sheafhom_{\cO_X}(F^e_*\cO_X(D-np^e D_0),\cO_X  )$, the map 
\[ \psi \star \tilde{\phi} (-nD_0) \in    F^{e+e'} _* (\fa_{p^{e'} -1} ^{[p^e]} \fa_{p^e -1}) \cdot( \sheafhom_{\cO_X}(F^{e+ e'}_*\cO_X(D),\cO_X) )\]
constructed in \autoref{def:globaloperations} splits $D$ and shows that $D \notin I_{e+e'} ^{\fad} (\cO_X (D))$ since $\fa_{p^{e' -1}} ^{[p^e]} \fa_{p^e -1} \subset \fa_{p^{e+ e'} - 1} $. This completes the proof of the Proposition. \qedhere
\end{proof}

\subsection{Cartier subalgebras over graded rings.}

Let $R$ be a Noetherian $F$-finite reduced ring. For given natural numbers $e,f\in \mathbb{N}$, consider $R$-module homomorphisms $\varphi \in \Hom_R(F^e_*R, R)$ and $\psi \in \Hom_R(F^f_*R, R)$. Now, we construct the star operation $\varphi \star \psi \in \Hom_R(F^{e+f}_* R, R)$ defined as a composed map $\varphi\star \psi := \varphi \circ F^e_*\psi$. Let $\cC _e (R) := \Hom_R(F^e_*R, R)$ and $\cC(R) = \bigoplus_{e\geq 0} \cC_e(R)$. Via the $\star$-opeation, $\cC(R)$ forms a non-commutative graded algebra over $R$.

\begin{dfn}
Let $R$ be a Noetherian $F$-finite reduced ring. The \emph{full Cartier algebra} $\mathcal{C}(R)$ is defined to be the non-commutative graded algebra
\[ \mathcal{C}(R) := \bigoplus_{e\geq 0} \cC_e =  \bigoplus_{e\geq 0} \Hom_R(F^e_*R,R).\]
A \emph{Cartier subalgebra over $R$}, denoted by $\cD\subseteq \cC(R)$ is a graded subalgebra of $\cC$ such that $\cD_0 = R$ (under addition and the $\star$-operation). Under the natural grading induced from $\cC (R)$, we denote the $e^{\text{th}}$-graded piece of $\cD$ by $\cD_e \subset \cC_e $ which is an $\Fe R$-submodule of $\Hom_R(F^e_*R,R)$.
\end{dfn}

The notions of $F$-regularity and the $F$-signature extend to the settings of Cartier subalgebras of an $F$-finite local rings as follows:

\begin{dfn}[\cite{BlickleSchwedeTuckerFSigPairs1}] \label{def:Cartiersubalgebras}
    Let $(R, \fm, k)$ be an $F$-finite local domain and $\cD$ be a Cartier subalgebra over $R$. Let $I_e ^{\cD}$ denote the following ideal of $R$:
\[ I_e^{\cD} = \{ f \in R \mid \varphi(F^e_* f) \in \fm \text{ for all } \varphi \in \cD_e \}. \]
The pair $(R, \cD)$ is said to be $F$-regular if for each non-zero element $c \in R$, there exists an $e \gg 0$ such that $c \notin I_e ^\cD$.
Next, assume that $\cD_e \neq 0$ for all $e \gg 0$. Then, the $F$-signature of the pair $(R, \cD)$ is defined as
\[ \s(R, \cD) := \lim_{e \to \infty } \frac{\ell (R/I_e ^\cD)}{p^{ed}} .\]
Here $\ell$ denotes the length as an $R$-module and $d$ denotes the dimension of $R$.
\end{dfn}

Moreover, we also have:
\begin{thm}{\cite[Theorems 3.11, 3.18]{BlickleSchwedeTuckerFSigPairs1}} Let $(R, \fm, k)$ be an $F$-finite local domain and $\cD$ be a Cartier subalgebra over $R$ such that $\cD_e \neq 0$ for all $e \gg 0$. Then, the $F$-signature of $(R, \cD)$ exists and it is positive if and only if the pair $(R, \cD)$ is $F$-regular.
    
\end{thm}

We will utilize this theorem by relating pairs $(X, \fad)$ where $X$ is a projective variety with an ample divisor $L$, with a suitable Cartier subalgebra over the section ring $S(X, L)$.

First we discuss the appropriate bi-grading on the full Cartier algebra of an $\NN$-graded ring $S$. Fix an integer $e \geq 0$. Since $S$ is an $\NN$-graded ring, we have that $\Fe S$ is naturally an $\frac{1}{p^e} \NN$-graded ring such that the map $F^e$ is a degree-preserving map. Moreover, for any $\ZZ$-graded $S$-module $N$, the module $ \Hom _S (\Fe S , N)$ is $\frac{1}{p^e}\ZZ$-graded with a decomposition given as
\begin{equation} \label{decompositionofHom} \Hom _S (\Fe S , N) =\bigoplus_{\frac{n}{p^e}\in \frac{1}{p^e}\ZZ}\Hom_S^{\text{gr}}\left(F^e_*S\left[\frac{-n}{p^e}\right], N\right) \end{equation}
where $F^e_*S\left[\frac{-n}{p^e}\right]$ is the module whose $\frac{m}{p^e}$ degree component (for any $m \in \ZZ$) is $(F^e_*S)_{\frac{m-n}{p^e}}$. Note that by definition, given a homogeneous element $f \in S_m$ and a homogeneous map $\phi \in \Hom_S^{\text{gr}}\left(F^e_*S\left[\frac{-n}{p^e}\right], N\right) $, the image $\phi (\Fe f) $ has degree $k$ if $m + n = k p^e $ (for some $k \in \ZZ$) and is equal to $0$ is if $m \not \equiv n \bmod p^e$.

\begin{dfn} \label{dfn:lambdagrading}
    Let $\frac{1}{p^\infty} \ZZ$ denote the additive group consisting of fractions of the form $\frac{n}{p^e}$ for some $e \geq 0$ and $n \in \ZZ$. We set $\Lambda : = \NN \times \frac{1}{p^\infty} \ZZ $. Then, by the above discussion, the full Cartier algebra of an $\NN$-graded ring $S$ is naturally a $\Lambda$-graded non-commutative algebra over $S$, which we denote as
    \[ \cC(S) = \bigoplus_{e \geq 0} \bigoplus _{\frac{n}{p^e} \in \frac{1}{p^e}\ZZ}\Hom_S^{\text{gr}}\left(F^e_*S\left[\frac{-n}{p^e}\right], S \right)  .\] We will consider $\Lambda$-graded Cartier subalgebras of $\cD \subset \cC$ with $\cD_0 = S$. Note that for such a subalgebra, each $e^{\text{th}}$-graded component $\cD_e$ is a $\frac{1}{p^e} \ZZ$-graded $\Fe S$-submodule of  $\Hom_S (\Fe S, S)$ and decomposes as
    \[  \cD_e = \bigoplus _{\frac{n}{p^e} \in \frac{1}{p^e}\ZZ} \Big ( \cD_{e,n} \subset \Hom_S^{\text{gr}}\left(F^e_*S\left[\frac{-n}{p^e}\right], S \right) \Big)  .\]
\end{dfn}

\begin{eg} \label{examplesoflambdagraded} There are several natural $\Lambda$-graded Cartier subalgebras that arise over an $\NN$-graded ring $S$.
    \begin{enumerate}
        \item Suppose $S$ is normal and $\Delta$ is an effective $\QQ$-divisor over $Y : = \Spec(S)$. Assume that $\Delta$ is \emph{homogeneous}, which means that the prime divisors occuring as components of the support of $\Delta$ are all defined by homogeneous prime ideals in $S$. Then, the Cartier subalgebra over $S$ associated to the pair $(Y, \Delta)$, denoted by $\cC(S, \Delta)$ is naturally $\Lambda$-graded as follows:
        \[ \cC(S, \Delta)_{e} = \Hom_S (\Fe ( S(\lceil (p^e -1) \Delta \rceil) ), S) =  \bigoplus _{\frac{n}{p^e} \in \frac{1}{p^e}\ZZ} \Hom_S^{\text{gr}}\left(F^e_*S (\lceil (p^e -1) \Delta \rceil) \, \left[\frac{-n}{p^e}\right], S \right). \]

        \item Continuing the previous example, suppose moreover that the $\NN$-graded ring $S$ is a section ring of a normal projective variety $X$ over $k$ with respect to an ample line bundle $\cL$. Let $\Delta$ be an effective $\QQ$-divisor over $X$ and $\tilde{\Delta}$ denote the cone over $\Delta$ with respect to $\cL$. Thus, $\tilde{\Delta}$ is a homogeneous $\QQ$-divisor on the cone $Y = \Spec (S(X, \cL))$. Then, the graded pieces of $\cC(S, \tilde{\Delta})$ can be understood as (see \autoref{lem.SectionRingVSCoherentSheaf})
        \[ \cC(S, \tilde{\Delta})_{e,n} = \Hom_S^{\text{gr}}\left(F^e_*S (\lceil (p^e -1) \tilde{\Delta} \rceil) \, \left[\frac{-n}{p^e}\right], S \right) =  \Hom_{\cO_X}
        (\Fe \cO_X (\lceil (p^ e-1) \Delta \rceil), \cL^n).  \]

        In particular, by taking $\Delta = 0$, we see that the full Cartier algebra $\cC (S)$ naturally consists of graded pieces 
        \[ \cC (S) _{e,n} = \Hom_{\cO_X} (\Fe \cO_X, \cL^n) = \Hom_{\cO_X} (\Fe \cL^{-np^e}, \cO_X) .  \]
        Moreover, under this identification of homogeneous maps, the $\star$-operation on $\cC(S)$ applied to homogeneous maps  matches with the global $\star$-operation on maps defined in \autoref{def:globaloperations}.
        
        \item Next, suppose $\fa$ is a homogeneous ($\NN$-graded) ideal of $S$. Then, $\Fe \fa$ is naturally $\frac{1}{p^e} \NN$-graded ideal of $\Fe  S$. Then, $ \Fe \fa \cdot \Hom_S (\Fe S, S) $ is a $\frac{1}{p^e} \ZZ$-graded $\Fe S$-submodule of  $\Hom_S (\Fe S, S)$ which we may compute as follows:
\begin{equation}\label{eqn.decomposition} \begin{split} F^e_* \fa \cdot \Hom_{S}(F^e_*S ,S)& = (F^e_* \fa) \cdot\bigoplus_{\frac{n}{p^e}\in \frac{1}{p^e}\ZZ}\Hom_S \left(F^e_*S, S\right)_{\frac{n}{p^e}} \\ & = \bigoplus _{\frac{m}{p^e} \in \frac{1}{p^e}\ZZ} \quad \sum _{ \ell\in \Z} (F^e_*\frak{a})_{\frac{\ell}{p^e}} \cdot \Hom_{S} (F^e_* S , S) _{\frac{m-\ell}{p^e}}.
\end{split}
\end{equation}
Thus, putting this together with \autoref{decompositionofHom}, we see that the graded components of $(\Fe \fa) \cdot \Hom_S (\Fe S, S)$ are given by
\begin{equation} \label{eqn:decompositioofhomwithideal} \left( (\Fe \fa) \cdot \Hom_S (\Fe S, S) \right)_{\frac{m}{p^e}}  = \sum _{ \ell\in \Z} (F^e_*\frak{a})_{\frac{\ell}{p^e}} \cdot \Hom_{S} ^{\mathrm{gr}}\left(F^e_* S \left[ \frac{-m+ \ell}{p^e} \right] , S\right).\end{equation}

Thus, given a graded sequence of homogeneous ideals $\fad$ over $S$ (or more generally, an $F$-graded sequence), there is a naturally $\Lambda$-graded Cartier subalgebra $\cC(S, \fad)$ defined as follows:
\[ \cC(S, \fad)_{e} = \left( (\Fe \fa_{p^e -1}) \cdot \Hom_S (\Fe S, S) \right)  \]
where the further decomposition of $\cC(S, \fad)_e = \bigoplus_{\frac{m}{p^e} \in \frac{1}{p^e} \ZZ } \,  \left( (\Fe \fa_{p^e -1}) \cdot \Hom_S (\Fe S, S) \right)_{\frac{m}{p^e}} $ can be computed using \autoref{eqn:decompositioofhomwithideal}.
    \end{enumerate}
\end{eg}

Given an $\NN$-graded ring $(S, \fm, k)$ (with $S_0 = k$), and a $\Lambda$-graded Cartier subalgebra $\cD$ over $S$, the ideals $I_e^\cD$ are also naturally $\NN$-graded with the following graded components:
\[ I_e ^\cD (m) = \{ f \in S_m \mid \varphi(F^e_*f) \in \fm \text{ for all homogeneous } \varphi \in \cD_{e} \}  . \]
Moreover, we have the following simpler description: 

\begin{lem} \label{lem:simplifiedIegraded}
    With notation as above, for any $e \geq 0$ and $m \geq 0$, we have 
    \[ I_e ^\cD (m) = \left\{f\in S_m \mid \varphi(F^e_*f) = 0 \text{ for all } \varphi \in \cD_{e,-m}  \right\} \]
\end{lem}

\begin{proof}
Note that the $\subseteq$ is straightforward since the right-hand side takes care of the homomorphisms of degree $\frac{-m}{p^e}$. For the converse direction, suppose $f \in S_m$ and $\varphi(F^e_*f) = 0 \text{ for all } \varphi \in \cD_{e,-m}$. Then for any homogeneous $\phi \in \cD_{e,n}$, we have that $\phi (\Fe f) $ is $0$ is $m \not \equiv n \bmod p^e $ or if $ n < -m$. But if $m + n = k p^e $ for some $k > 0$, then $\phi (\Fe f)$ has degree $k>0$ and hence already lies in $\fm$. Therefore, we see that $\phi (\Fe f) \in \fm$ in any case and therefore for all $\phi \in \cD_e$, which shows that $f \in I_e ^\cD$.
\end{proof}


\subsection{Cartier subalgebra associated to a graded sequence.}
In this subsection, we will complete the proof of \autoref{thm:fsigofgradedsequences}. For this, we will associate to a graded sequence of ideal sheaves, an appropriate Cartier subalgebra over the section ring:

\begin{dfn} \label{def:Cartiersubalgebraassociatedtogradedsequence}
    Let $X$ be a normal projective variety and $\cL$ be an ample line bundle on $X$. Let $S=S(X,\cL)$ denote the section ring of $X$ with respect to $\cL$. Then, given a graded sequence of coherent ideal sheaves $\fad$ on $X$ (\autoref{gradedsequencedfn}), we define the associated Cartier subalgebra $\cD(S, \fad) = \cD(X, \cL, \fad)$ over $S$ as:
    \[ \cD(X, \cL, \fad)_e = \bigoplus _{\frac{m}{p^e} \in \frac{1}{p^e} \ZZ} \Big ( \cD_{e,m} :=  H^0 \big( X, \Fe \fa_{p^e-1} \cdot \sheafhom_{\cO_X}(F^e_*\cL^{-m},\cO_X)  \big)  \Big).  \]
    It follows from the compatibility of the global $\star$-operation defined in \autoref{def:globaloperations} and the $\star$-operation on $\cC(S)$ (see \autoref{examplesoflambdagraded}, (b)) that $\cD(S, \fad)$ is indeed a Cartier subalgebra over $S$.
    
\end{dfn}

The key observation to prove \autoref{thm:fsigofgradedsequences} is the following lemma relating the $I_e$-ideals:

\begin{lem} \label{lem:relatingIeofgradedsequenceandCartiersubalgebra}
  Let $X$ be a normal projective variety and $\cL$ be an ample line bundle on $X$ and $S=S(X,\cL)$ denote the section ring of $X$ with respect to $\cL$. Suppose we have a graded sequence of coherent ideal sheaves $\fad$ on $X$ (\autoref{gradedsequencedfn}) and the associated Cartier subalgebra $\cD(S, \fad) = \cD(X, \cL, \fad)$ over $S$ (\autoref{def:Cartiersubalgebraassociatedtogradedsequence}). Then, for all $e \geq 1$ and $m \geq 0$, we have
  \[ I_e ^{\fad} (\cL^m) = I_e ^\cD (m).  \]
  Here $I_e ^{\fad} (\cL^m) $ is calculated over $X$ as in \autoref{def:Iebullet} and $I_e ^\cD$ is calculated over $S$ as in \autoref{lem:simplifiedIegraded}.
\end{lem}

\begin{proof} 
Note that $S_m = H^0 (X, \cL^m)$ so both  $I_e ^{\fad} (\cL^m)$ and $ I_e ^\cD (m)$ are subspaces of $S_m$. By \autoref{lem:simplifiedIegraded}, a homogeneous element $f \in S_m$ is contained in $I_e ^\cD (m)$ if and only if $ \varphi (\Fe f) = 0 $ for all $\varphi \in \cD_{e, -m} = H^0 \big( X, \Fe \fa_{p^e-1} \cdot \sheafhom_{\cO_X}(F^e_*\cL^{m},\cO_X)  \big)  \Big) $. But given a map $\varphi \in \Hom _{\cO_X}( \Fe \cL^m ,  \cO_X)$, the corresponding map $\tilde{\varphi}$ in the Cartier algebra of $S$ is just the cone over the map $\varphi$. Therefore, we have $\varphi (\Fe (1_f)) = 0$ if and only $\tilde{\varphi}(f) = 0 $, where we view $ 1_f $ the canonical global section of $\cL^m$ defined by $f \in S_m$. This proves the lemma.
\end{proof}

\begin{proof}[Proof of \autoref{thm:fsigofgradedsequences}]
    We set $\cD$ to be the Cartier subalgebra over $S$ associated to $(X, \cL, \fad)$ as in \autoref{def:Cartiersubalgebraassociatedtogradedsequence}.
    
    By the definition of the $F$-signature $\s_X(\fa_\bullet ,L)$ (\autoref{def:fsigofidealsequence}) and \autoref{lem:relatingIeofgradedsequenceandCartiersubalgebra}, we have the equivalent limit 
    \[ \s_X(\fad,L) = \lim_{e \to\infty} \frac{1}{p^{e(d+1)}}\sum_{r=0}^\infty \dim \frac{S_r}{I^{\cD} _e (r)} \]
    which exists by \cite[Theorem 3.11]{BlickleSchwedeTuckerFSigPairs1}. Note that for any $e \geq 1$, the space $\cD_{e,n} = H^0 \big( X, \Fe \fa_{p^e-1} \cdot \sheafhom_{\cO_X}(F^e_*\cL^{-n},\cO_X)  \big) \neq 0 $ for all $n \gg 0$ by the ampleness of $\cL$. $\cD_e \neq 0$ for all $ e \geq 0$.

    Next, by \cite[Theorem 3.18]{BlickleSchwedeTuckerFSigPairs1} the positivity of the $F$-signature $\s (S, \cD)$ is equivalent to the $F$-regularity of the pair $(S, \cD)$ (see \autoref{def:Cartiersubalgebras}). Therefore, it suffices to check that the $F$-regularity of the pair $(S, \cD)$ is equivalent to the global $F$-regularity of the pair $(X, \fad)$. Suppose $(S, \cD)$ is $F$-regular, which by definition means that for each non-zero element $c \in S$, there is some $e \gg 0$ such that $c \notin I_e ^\cD$. Applying this to homogeneous elements of $S$ and using \autoref{lem:relatingIeofgradedsequenceandCartiersubalgebra}, we see that for any divisor $D \sim rL$ (for some $r >0$), there exists an $e \gg 0$ such that $D \notin I_e ^{\fad} (rL)$. By \autoref{prop:testdivisor}, this implies that the pair $(X, \fad)$ is globally $F$-regular. Conversely, suppose $(X, \fad)$ is globally $F$-regular but $(S, \cD)$ is not $F$-regular. Then, there exists a non-zero element $c \in S$ such that $c \in I_e ^\cD$ for all $e \geq 1$. But since the ideals $I_e ^\cD$ are homogeneous, this implies that if $c = \sum c_r$ is the homogeneous decomposition of $c$, then each $c_r \in I_e ^\cD$ for all $e \geq 1$. But again, using \autoref{lem:relatingIeofgradedsequenceandCartiersubalgebra}, and the global $F$-regularity of $(X, \fad)$, we know that $c_r \notin I_e ^\cD$ for $ e \gg 0$. This is a contradiction and shows that $(S, \cD)$ is $F$-regular. Therefore, by \cite[Theorem 3.18]{BlickleSchwedeTuckerFSigPairs1}, the pair $(X, \fad)$ is globally $F$-regular if and only if the $F$-signature $\s_X (\fad, L)$ is positive. This completes the proof of \autoref{thm:fsigofgradedsequences}.
\end{proof}



\begin{dfn}
    Let $X$ be a projective variety and $\cL$ be an ample line bundle on $X$. Let $S=S(X,\cL)$ denote the section ring of $X$ with respect to $\cL$. Then, given a graded sequence of coherent ideal sheaves $\fad$ on $X$,
    the \emph{cone over $\fad$} is a graded sequence of ideals $\tfad$ of $S$ defined by
\[  \tilde{\fa}_n = M(\fa_n, L) := \bigoplus _{r \geq 0} H^0 (X, \fa_n \otimes L^r). \]
We note that since $\fa_m \cdot \fa_n \subset \fa_{m+n} $, the cones $\tfad$ are automatically a graded sequence since the condition can be checked on homogeneous elements and the image of the map $H^0 (X, \fa_m \otimes L^r) \otimes H^0 (X, \fa_n \otimes L^s) \to H^0 (X, L^{r + s} ) $ lies in $H^0 (X, \fa_{m+n} \otimes L^{r+s})$.
\end{dfn}

\section{Remarks about pairs and triples.} \label{section:pairsandtriples}
\subsection{Remarks about the $(X, \Delta)$ case.}

Let $X$ be a normal projective variety over $k$. Let $\Delta$ be an effective $\bQ$-Weil divisor on $X$. For any Weil divisor $D$ on $X$ and $e\geq 1$, define the $k$-vector subspace $I^\Delta_e(D)$ of $H^0
(X, \cO_X(D))$ as follows:
\[I_e^\Delta(D) = \{ s \in H^0(D) \mid \varphi(F^e_* s) = 0 \text{ for all } \varphi \in \Hom_{\cO_X}(F^e_*\cO_X(\lceil (p^e-1)\Delta \rceil +D), \cO_X)\}.\]
Now, let $L$ be a big divisor on $X$. We define the $F$-signature of a pair

\begin{dfn}\label{def:FsigPair}
Let $X$ be a $d$-dimensional normal variety over $k$. Let $L$ be a big divisor on $X$. The \emph{$F$-signature function of a pair $\s_X(\Delta,L)$} is \[\s_X(\Delta, L) = \lim_{e\to\infty}\frac{1}{p^{e(d+1)}}\sum_{m=0}^\infty\dim_k \frac{H^0(mL)}{I_e^\Delta(mL)}\]
\end{dfn}

\begin{thm}
Let $X$ a normal projective variety and $L$ be a big Cartier divisor. Then,
\begin{enumerate}
    \item The $F$-signature of a pair $s_X(\Delta,  L)$ exists.
    \item (Scaling property) For any $n\in \mathbb{N}$, $s_X(\Delta, nL) = \frac{1}{n} s_X(\Delta,  L)$.
    \item The function $L \mapsto s_X(\Delta, L)$ is locally Lipschitz continuous on the big cone.
\end{enumerate}
\end{thm}

\begin{proof}
The key ingredients in establishing the first two properties are \autoref{Pn:twistedFsig} and \autoref{Pn:TransformationRuleFsig}. The arguments adapt directly to $I_e^\Delta$ since the essential part of the proof is comparing the dimensions of the spaces. As a consequence, $\s_X(\Delta,L)$ is locally Lipschitz continuous on the big cone of $X$.
\end{proof}

Now, we define the Frobenius $\alpha$-invariant of a divisor pair. 
\begin{dfn} \label{alphae}
 For each integer $e \geq 1$, we define
\[ m_{e} (X,\Delta , L) : = \max \{ m \geq 0 \, | \, I_{e}^\Delta (mL)  = 0 \}, \]
where $I_e^\Delta(L)$ is the subspace defined as above. We also define
\[ \alpha_{e}(X,\Delta, L) := \frac{m_e (X,\Delta, L)}{p^e} .\] The \emph{$\alpha_F$-invariant of pairs} is $\alpha(X,\Delta,L)= \lim_{e\to\infty} \alpha_e(X,\Delta,L)$. 
\end{dfn}

\begin{Pn}\label{Pn:AlphaInvPairsPropositions}
Let $X$ be a normal projective variety and $L$ be an integral big divisor on $X$. Let $\Delta$ be an effective $\bQ$-divisor on $X$. Then,
\begin{enumerate}
\item the $\alpha_F$-invariant $\alpha(X,\Delta,L)$ exists.
\item for any $n\in \mathbb{N}$, $\alpha_F(X,\Delta, nL) = \frac{1}{n}\alpha_F(X,\Delta, L)$.

\end{enumerate}

\end{Pn}
\begin{proof}
The first part is similar to  \autoref{lem:AlphaInvExists} together with \cite[Lemma 4.5]{PandeFrobeniusVersionTiansAlphaInvariant} to deal with the difference of rounding. The second follows by an analogous argument using \autoref{Prop:transformationforalpha}, since the proof relies on the index $m_e$ that makes $I_e^\Delta$ trivial and the linear spaces of big divisors $|mL|$ for $m\gg0$.
\end{proof}
If $L$ is a big $\bQ$-Cartier $\bQ$-divisor, there is a natural number $n\in\mathbb{N}$ such that $nL$ is an integral big divisor. In this case, we define $\alpha_F(X,\Delta, L):= n\alpha_F (X,\Delta, nL)$. Note that $\alpha_F(X,\Delta, L)$ is independent of the choice of $n$.

\begin{Pn}
Let $X$ be a normal projective variety and $L$ be a big $\bQ$-Cartier divisor on $X$. Let $\Delta$ be an effective $\bQ$-divisor on $X$. Then, \begin{enumerate}
\item for any effective $\bQ$-divisor $E$, $\alpha_F(X, \Delta, L) \geq \alpha_F(X, \Delta, L+E)$.
\item the alpha invariant is locally Lipschitz continuous on the big cone.
\end{enumerate}
\end{Pn}
\begin{proof}
For the first part, let $n$ and $m$ be the Cartier indices of $L$ and $E$, respectively, and set $L' = mnL$ and $E' = mnE$. Since $I_e^\Delta(L') \subseteq I_e^\Delta(L' + E')$, it follows that $m_e(X, \Delta, L') \geq m_e(X, \Delta, L' + E')$. Together with \autoref{Pn:AlphaInvPairsPropositions} and the definition of $\alpha_F$, the desired inequality follows. For the second part, recall that the continuity of the non-pair version of $\alpha_F$ relies on \autoref{Prop:addingeffective} and \autoref{lem:upperboundonalpha}. The first part above provides the paired analogue of \autoref{Prop:addingeffective}. Moreover, by definition, we have $\alpha_F(X, \Delta, L) \leq \alpha_F(X, L)$, so $\alpha_F(X, \Delta, L)$ satisfies the same upper bound as in \autoref{lem:upperboundonalpha}. Therefore, by running the same argument as in \autoref{Prop:alphacontinuity}, we obtain the continuity of the alpha invariant of the pair.
\end{proof}

\subsection{Remarks about triples $(X,\Delta, \frak{a}_\bullet)$}\label{subsec:triples}  In this subsection, we combine the theory of pairs and graded sequences. Let $X$ be a normal projective variety over $k$. A \emph{triple} $(X, \Delta, \mathfrak{a}_{\bullet})$ consists of the variety $X$, an effective $\mathbb{Q}$-divisor $\Delta$, and a graded sequence of ideal sheaves $\mathfrak{a}_{\bullet}$ (see \autoref{gradedsequencedfn}). We generalize the definition of the $I_e$ space to this setting. For any Weil divisor $D$ on $X$ and integer $e \geq 1$, we define the subspace $I_e^{\Delta, \mathfrak{a}_{\bullet}}(D)$ of $H^0(X, \cO_X(D))$ as follows:
\[
I_e^{\Delta, \mathfrak{a}_{\bullet}}(D) = \left\{ s \in H^0(D) \;\middle|\; \varphi(F^e_* s) = 0 \text{ for all } \varphi \in H^0(\mathcal{F}_{e,\Delta, \mathfrak{a}_{\bullet}}(D)) \right\}
\]
where $\mathcal{F}_{e,\Delta, \mathfrak{a}_{\bullet}}(D)$ denotes the sheaf of splitting maps compatible with the triple:
\[
\mathcal{F}_{e,\Delta, \mathfrak{a}_{\bullet}}(D) := \mathfrak{a}_{p^e-1} \cdot \mathscr{H}om_{\cO_X} \big( F^e_* \cO_X(\lceil (p^e-1)\Delta \rceil + D), \cO_X \big) .
\]
We say that the triple $(X,\Delta,\mathfrak{a}_\bullet)$ is
\emph{globally $F$-regular} if, for every effective Weil divisor
$D$ on $X$, there exists an integer $e>0$ such that the canonical
section $1\in H^0(X,\cO_X(D))$ does not belong to
$I_e^{\Delta,\mathfrak{a}_\bullet}(D)$.
\begin{dfn} \label{def:FsigTriple}
Let $X$ be a $d$-dimensional normal projective variety. Let $(X, \Delta, \mathfrak{a}_{\bullet})$ be a triple as above and $L$ be a big divisor on $X$. We define the \emph{F-signature function of the triple}, denoted $s_X(\Delta, \mathfrak{a}_{\bullet}, L)$, as:
\[
s_X(\Delta, \mathfrak{a}_{\bullet}, L) = \lim_{e \to \infty} \frac{1}{p^{e(d+1)}} \sum_{m=0}^{\infty} \dim_k \frac{H^0(mL)}{I_e^{\Delta, \mathfrak{a}_{\bullet}}(mL)}.
\]
\end{dfn}

\begin{thm}
Let $X$ a normal projective variety and $L$ be a big Cartier divisor. Let $\frak{a}_\bullet$ be a graded sequence of ideal sheaves on $X$. Then,
\begin{enumerate}
    \item The $F$-signature of triple $s_X(\Delta, \mathfrak{a}_{\bullet}, L)$ exists.
    \item (Scaling property) For any $n\in \mathbb{N}$, $s_X(\Delta, \mathfrak{a}_{\bullet}, nL) = \frac{1}{n} s_X(\Delta, \mathfrak{a}_{\bullet}, L)$.
    \item The function $L \mapsto s_X(\Delta, \mathfrak{a}_{\bullet}, L)$ is locally Lipschitz continuous on the big cone.
\end{enumerate}
\end{thm}
\begin{proof}
The proof is identical to the pair version when $\frak{a}_\bullet$ is trivial. Hence, we omit the proof.
\end{proof}

We may also define the Frobenius alpha invariant for triples.

\begin{dfn}
For a triple $(X, \Delta, \mathfrak{a}_{\bullet})$ and a big divisor $L$, we define:
\[ m_e(X, \Delta, \mathfrak{a}_{\bullet}, L) := \max \{ m \ge 0 \mid I_e^{\Delta, \mathfrak{a}_{\bullet}}(mL) = 0 \}. \]
The \emph{$\alpha_F$-invariant of the triple} is defined as $\alpha_F(X, \Delta, \mathfrak{a}_{\bullet}, L) := \lim_{e \to \infty} \frac{m_e(X, \Delta, \mathfrak{a}_{\bullet}, L)}{p^e}$.
\end{dfn}

Finally, we list analogous properties for alpha invariant of triples:

\begin{Pn}\label{Pn:AlphaInvPairsPropositions}
Let $(X,\Delta,\frak{a}_\bullet) $ be a triple and $L$ be an integral big divisor on $X$. Then,
\begin{enumerate}
\item the $\alpha_F$-invariant $\alpha(X,\Delta,\frak{a}_\bullet,L)$ exists.
\item for any $n\in \mathbb{N}$, $\alpha_F(X,\Delta,\frak{a}_\bullet, nL) = \frac{1}{n}\alpha_F(X,\Delta,\frak{a}_\bullet, L)$.

\end{enumerate}

\end{Pn}

If $L$ is a big $\bQ$-Cartier $\bQ$-divisor, there is a natural number $n\in\mathbb{N}$ such that $nL$ is an integral big divisor. In this case, we define $\alpha_F(X,\Delta,\frak{a}_\bullet, L):= n\alpha_F (X,\Delta,\frak{a}_\bullet, nL)$. Note that $\alpha_F(X,\Delta, L)$ is independent of the choice of $n$.

\begin{Pn}
Let $(X,\Delta, \frak{a}_\bullet)$ be a triple and $L$ be a big $\bQ$-Cartier divisor on $X$. Then, \begin{enumerate}
\item for any effective $\bQ$-divisor $E$, $\alpha_F(X,\Delta,\frak{a}_\bullet, L) \geq \alpha_F(X,\Delta,\frak{a}_\bullet, L+E)$.
\item the alpha invariant is locally Lipschitz continuous on the big cone.
\end{enumerate}
\end{Pn}

\section{Transformation rules for the $F$-signature along proper birational maps} \label{section:transformationrule}

In this section, we explore how the $F$-signature functions transform under proper birational maps. This provides a geometric interpretation for the value of the $F$-signature function $\s_X(L)$ for a \emph{big and semi-ample} line bundle $L$ on $X$ in terms of the contraction defined by (multiples of) $L$. Recall that the $F$-signature function of big and nef divisors was defined by continuous extension from the ample cone, or equivalently by \autoref{Dfn.FsigBigDivisors}.

\begin{thm}\label{Thm:transformationRuleProperBirational}
    Let $X$ be a globally $F$-regular projective variety and $L$ be a big and semi-ample divisor over $X$. Suppose $\pi: X \to Y$ is the birational contraction defined by $|mL|$ for some $ m \gg 0$ sufficiently divisible and $L_0$ denote the ample divisor on $Y$ such that $\pi ^* L_0 \sim L$. Then, we have
    \begin{enumerate}
        \item Suppose $K_Y $ is $\QQ$-Cartier and $K_{X|Y} = K_X - \pi^* K_Y \leq 0$. Then,  $\s_X (L) = \s_Y (L_0) $.

        \item Suppose $K_Y $ is $\QQ$-Cartier and we have $K_{X|Y} = E \geq 0$ for some effective $\QQ$-divisor $E$. Then, we have a natural graded sequence of ideal sheaves over $Y$ defined by
        \[ \fa_m := \pi_* \cO_X (\lfloor -mE \rfloor ).  \]
        Moreover, we have the following transformation rule for the $F$-signature:
        \[ \s_X (L) = \s_Y (\fad, L_0). \]
        
        \item For any effective $\QQ$-divisor $\Delta_X $ such that $(X, \Delta_X)$ is globally $F$-regular and $-K_X - \Delta_X  \sim _\QQ L$, we have:
        \begin{itemize}
            \item setting $\Delta_Y = \pi_* \Delta_X $, the pair $(Y, \Delta_Y)$ is globally $F$-regular.

            \item the divisor $- K_Y - \Delta_Y$ is $\QQ$-Cartier and $ - K_Y - \Delta_Y \sim _\QQ L_0 $.

            \item We have 
            \[ \s_X (\Delta_X, L) = \s_Y (\Delta_Y, L_0).\]
        \end{itemize}
        
    \end{enumerate}
Here $\s_Y(\fad,L_0)$ is the $F$-signature defined in \autoref{def:fsigofidealsequence}, while the $F$-signatures
of divisor pairs are defined in \autoref{def:FsigPair}.

\end{thm}

\begin{eg}
We provide an example of a smooth blow-up for \autoref{Thm:transformationRuleProperBirational}. In fact, it provides an asymptotic generalization of \cite[Proposition 2.1]{LakshmibaiFrobSplittingsandBlowups}. Let $X$ be a smooth globally $F$-regular projective variety and let $Z\subset X$ be a smooth subvariety of codimension $d\geq 2$. 
Consider the blow-up $\pi:Y = \mathrm{Bl}_Z(X) \to X $ along $Z$. Let $E$ be the exceptional divisor. Then, the relative canonical $K_{Y|X} = (d-1)E$ is effective. Let $\frak{a}_m = \pi_*\cO_Y(\lfloor -m(d-1)E\rfloor)$. Then, for any ample divisor $L$ on $X$, we have $\s_Y(\pi^*L) = \s_X(\frak{a}_\bullet, L)$ by \autoref{Thm:transformationRuleProperBirational}.

\end{eg}

This theorem is a consequence of the compatibility of the trace-map of the Frobenius under proper and birational maps. We begin by reviewing the relevant facts about the trace map.


\subsection{The trace map.} Let $X$ be a normal projective variety over $k$. Then, for each $e \geq 1$, we have a canonically defined trace map:
\[ \text{tr}^e : \Fe (\cO_X((1-p^e)K_X))) \to \cO_X.  \]
Let $\Tr^e$ denotes the map on the global sections induced by $\tr$. Moreover, given an effective $\QQ$-divisor $\Delta$ over $X$, let $\text{tr} ^e _\Delta$ denote the natural composition map
\[ \Fe \Big (\cO_X  \big ((1-p^e)K_X) - \lceil (p^e  -1 ) \Delta \rceil \big )  \Big ) \to  \Fe \big( \cO_X((1-p^e)K_X) \big) \to  \cO_X \]
 and $\text{Tr} ^e _{\Delta} $ the map induced by $\text{tr}^e _{\Delta}$ on the global sections.

The map $\Tr^e$ satisfies the following useful property: for any effective Weil divisor $D \geq 0$, the map $\cO_X \to \Fe (\cO_X(D))$ splits if and only if there exists an effective divisor $D' \sim (1-p^e)K_X - D$ such that $\Tr^e (\Fe  (D+D')) \neq 0.$
Note that we may think of $D+D'$ as a global section of the sheaf $\cO_X((1-p^e)K_X)$ since both $D$ and $D'$ are effective. More generally, we have:

\begin{lem} \label{lem:multiplicationandtracemap}
    Let $X$ be a normal projective variety, $\Delta$ an effective $\QQ$-divisor and $\fad $ a graded sequence of ideal sheaves over $X$ (\autoref{gradedsequencedfn}). Then, for any $e \geq 1$ and any Weil-divisor $D $, let $\cF_{e, \Delta, \fad} (D)$ denote the sheaf $ \fa_{p^e  -1 } \cdot (\cO_X( (1-p^e)K_X -  \lceil (p^e - 1) \Delta \rceil )  -D )$ over $X$. Then,  
    the subspace $I_e ^{\Delta, \fad} (D)$ can be described as follows:
    \[ I_e ^{\Delta, \fad} (D) = \{ f \in H^0 (\cO_X (D)) \, | \, \text{Tr} ^e _\Delta (\Fe (fg)) = 0 \text{ for all } g \in H^0 (\cF_{e, \Delta, \fad} (D)) \}. \]
\end{lem}

\begin{proof}
    Note that by the duality for the Frobenius map, we have an isomorphism
\[ \mathscr{H}om_{\cO_X}(F^e_*\cO_X(\lceil (p^e-1)\Delta \rceil +D), \cO_X) \isom \Fe (\cO_X( (1-p^e)K_X -  \lceil (p^e - 1) \Delta \rceil )  -D ).   \]
Via this isomorphism, we may identify any section $g \in H^0 (\cF_{e,\Delta, \fad} (D)$ with a section of $    H^0(\Fe\frak{a}_{p^e-1} \cdot\sheafhom_{\cO_X}(F^e_*\cO_X(\lceil (p^e-1)\Delta \rceil + D),\cO_X))$
which is in turn a subspace of 
\[ H^0 (X, \Fe (\cO_X( (1-p^e)K_X -  \lceil (p^e - 1) \Delta \rceil )  -D ) ).\] Now the lemma follows from
\cite[Lemma 2.17]{PandeFrobeniusVersionTiansAlphaInvariant}
and the definition of $I_e^{\Delta,\fad}(D)$ in
\autoref{subsec:triples}.
\end{proof}

Now suppose we have a proper birational map $\pi: X \to Y$ between normal varieties $X,Y$ over $k$. Assume that $K_Y$ is $\QQ$-Cartier. Then, there is a corresponding graded sequence of ideals on $Y$ defined as follows: first we write $K_X - \pi^* K_Y = E$ for some $\pi$-exceptional $\QQ$-divisor $E$ on $X$ (we can do this by pulling back a multiple $rK_Y$ for some $r>0$ such that $rK_Y$ is Cartier). Then, we define
\[ \fa ^\pi _m = \pi_* \cO_X(\lfloor -mE \rfloor ).  \]

Let us check that $\fa_m ^\pi \subset \cO_Y$ for any $m \geq 0$. Write $E = E_1 - E_2$ for effective, $\pi$-exceptional $\QQ$-divisors $E_1$ and $E_2$ with no common components. Then we have 
\[  \fa ^\pi _m = \pi_* \cO_X(\lfloor -mE \rfloor ) =  \pi_* \cO_X(\lfloor -m E_1 \rfloor + \lfloor m E_2 \rfloor ) = \pi_* \cO_X(\lfloor -m E_1 \rfloor) \subset \pi_* \cO_X = \cO_Y,\]
where we have used the fact that $\pi_* \cO_X (F) = 0$ for any $\pi$-exceptional and effective divisor $F$ on $Y$. Moreover, if we have

In particular, for any such $e \geq 1$,  $\pi_*$ induces a map
\[ \pi_*: H^0 (Y, \cO_Y ((1-p^e)K_Y)) \to H^0 (X, \cO_X((1-p^e)K_X)). \]
 let us write $K_Y = \pi^* K_X + E_1 - E_2$ for effective, exceptional $\QQ$-divisors $E_1$ and $E_2$ (no common components). Then, we have $(1-p^e)K_Y = \pi^* ((1-p^e)K_X) - (p^e -1)E_1 + (p^e - 1)E_2$. Note that $(p^e -1)E_i$ should be $\ZZ$-divisors. Then, using the projection formula and the fact that $E_1$ and $E_2$ are effective, $\pi_*$ induces an inclusion
\[\xymatrix{
\pi_* : H^0 (Y, \cO_Y((1-p^e)K_Y)) \ar[r]& \ar@{=}[d]H^0 (Y, \cO_Y( \pi^* (1-p^e)K_X)  + (p^e -1) E_2)) \\ &  H^0 (X,  \cO_X((1-p^e)K_X))
}\]
Note that in particular, if $E_1 = 0$, in other words, if $K_Y \leq \pi^* K_X $, then $\pi_*$ is an isomorphism.

 We also have following compatibility of the trace map with respect to the map $\pi$:
\begin{lem}\label{lem.PushDownComparison}
    With notation as above, for each $e \geq 1$ such that $(1-p^e) K_X$ is Cartier, we have the following commutative diagram:
    \[
\xymatrix{
H^0(\Fe \cO_Y((1-p^e)K_Y)) \ar[d]^{\pi_*}  \ar[r]^<<<<{\Tr ^e} & H^0 (Y, \cO_Y) = k\ar[d]^{\pi_*}_\sim \\
 H^0(F^e_*\cO_X((1-p^e)K_X))\ar[r]\ar[r]^<<<<{\Tr ^e}& H^0 (X, \cO_X) = k 
}
\]
\end{lem}
\begin{proof}
This follows from the definition of the trace-map, which, under duality for the Frobenius map, is identified with the evaluation at $\Fe 1$ map
\[  \mathscr{H}om _{\cO_X} (\Fe \cO_X, \cO_X) \to \cO_X  .\]
\end{proof}

\subsection{Birational Tranformation rules for $F$-signature.}

\begin{thm}\label{thm.GlobalFregularityProperBirational}
Let $\pi: X \to Y$ be a proper birational map between normal, projective varieties over $k$. Suppose we have an effective $\QQ$-divisor $\Delta_Y$ on $Y$ such that $K_Y + \Delta_Y$ is $\QQ$-Cartier. Write $K_ X + \Delta_X = \pi^* (K_Y + \Delta_Y) + E$, where $\Delta_X$ is the strict transform of $\Delta_Y$ along $Y$ and $E$ denotes some $\pi$-exceptional $\QQ$-divisor over $X$. Let $\fa_\bullet$ denote the graded sequence of ideal sheaves defined by
\[ \fa_m = \pi_* \cO_X(\lfloor -mE \rfloor ) \]
for any $m \geq 1$.
Then, we have
\begin{enumerate}
    \item $(X, \Delta_X)$ is globally $F$-regular if and only if the triple $(Y, \Delta_Y, \fa_\bullet)$ is globally $F$-regular.
    \item For any ample divisor $L_0$ on $Y$, we have the equality of $F$-signatures
    \[ \s_X( \Delta_X, \pi^* L_0) = \s_Y( \Delta_Y, \fa_\bullet ,L_0). \]
\end{enumerate}
\end{thm}
Here the global $F$-regularity of the triple
$(Y,\Delta_Y,\fa_\bullet)$ is defined in the sense of
\autoref{subsec:triples}, and
$\s_Y(\Delta_Y,\fa_\bullet,L_0)$ denotes its $F$-signature
as defined in \autoref{def:FsigTriple}.

\begin{proof}
Both parts rely on the following key observation: for any $e \geq 1$ and any Cartier divisor $D$, we have an equality
\begin{equation} \label{eqn:Ieequality} I_e ^{\Delta_X} (\pi^* D) = \pi^* (I_e ^{\Delta_Y, \fad} (D) ).   \end{equation}
 Let $\cF_{e, \Delta_Y} (Y, D)$ denote the sheaf $ \fa_{p^e  -1 } \cdot (\cO_Y( (1-p^e)K_Y -  \lceil (p^e - 1) \Delta_X \rceil )  -D )$ on $Y$ and $\cD_{e, \Delta_X} ( X,D)$ the divisor $(1-p^e)K_X - \lceil (p^e  -1 ) \Delta_X \rceil - \pi^* D)$ on $X$. Following \autoref{lem.PushDownComparison}, let us consider the following commutative diagram
\[ \begin{tikzcd}
H^0(\Fe \cO_X(\pi^*D)) \ar[d, "\pi_*"] & [-3em] \times & [-3em] H^0(\Fe \cO_X(\cD_{e, \Delta_X} (X, D)) ) \ar[d, "\pi_*"] \ar[r,"\Fe \cdot"] & [-1em] H^0 (\Fe \cO_X( \cD_{e, \Delta_X} (X,0) )) \ar[d, "\pi_*"]    \ar[r, "\Tr^e _{\Delta_X}"] & [1em] k \ar[d, " \isom "] \\
H^0(\Fe \cO_Y(D))  &  \times &  H^0 ( \Fe \cF_{e, \Delta} (Y, D) )   \ar[r, "\Fe \cdot"] &  H^0 (\Fe \cO_X( \Fe \cF_{e, \Delta_Y} (Y, 0) ) ) \ar[r, "\Tr^e _{\Delta_Y}"] & k 
\end{tikzcd} \]
where the horizontal maps $\Fe \cdot $ are induced by multiplication of sections (before applying $\Fe$) as in \autoref{lem:multiplicationandtracemap}. Now, since we have $(1-p^e)K_X - \lceil (p^e  -1 ) \Delta_X \rceil  = \pi^* \big ( (1-p^e)K_Y - \lceil (p^e  -1 ) \Delta_Y \rceil  \big ) -\lfloor (1-p^e)E \rfloor $, we see that
\[
\begin{aligned}
\pi_* H^0(X, \cO_X(\cD_{e, \Delta_X} (X, D))) &= \pi_* H^0(X, \pi^* \big ( (1-p^e)K_Y - \lceil (p^e  -1 ) \Delta_Y \rceil  \big ) -\lfloor (1-p^e)E \rfloor )\\
  &= H^0(Y , \cF_{e, \Delta_Y} (Y, D) ) \text{\hskip 20pt $(\ast)$}
  \end{aligned} \]
by the projection formula. Therefore, all the vertical maps in the above diagram are isomorphisms.
By \autoref{lem:multiplicationandtracemap}, a section $f \in H^0 (Y, \cO_Y (D))  $ is contained in $I_e ^{\Delta_Y, \fad} (D) $ if and only if for all $ g \in H^0 (\cF_{e, \Delta_Y} (Y,D) ) $, we have $\text{Tr}^e _{\Delta_Y} (\Fe (fg)) = 0 $. Similarly, $\pi^* (f) \in I_e ^{\Delta_X} (\pi^* D) $ if and only if for all $g \in H^0 (X, \cO_X (\cO_X (\cD_{e, \Delta_X} (X,D)) ) )$, we have $\text{Tr} ^e _{\Delta_X} (\Fe (fg))  = 0$. Therefore, the equality in \autoref{eqn:Ieequality} follows from the commutativity of the above diagram and the fact that the vertical maps are isomorphisms.

For (a), suppose first that $(X, \Delta_X)$ is globally $F$-regular and we wish to show that $(Y, \Delta_Y, \fad)$ is globally $F$-regular. Let $D$ be an effective Cartier divisor on $Y$. Since $(X, \Delta_X)$ is globally $F$-regular, we know that for all $e \gg 0$, the global section $1 = \pi ^* 1 \in H^ 0 (X, \cO_X (\pi^* (D)))$ is not contained in $I_e ^{\Delta_X} (\pi^* (D))$. Therefore, by \autoref{eqn:Ieequality}, $1 \notin I_e ^{\Delta_Y, \fad}$. Since $D$ was arbitrary, we conclude that $(Y, \Delta_Y, \fad)$ is globally $F$-regular.

Conversely, suppose $(Y, \Delta_Y, \fad)$ is globally $F$-regular and we need to show that so is $(X, \Delta_X)$. Fix a very ample divisor $L$ on $X$. By \autoref{prop:testdivisor}, it suffices to show that for any effective divisor $D \sim L$ on $X$, we have $D \notin I_e ^{\Delta_X} (L)$ for $ e \gg 0$. We may choose a sufficiently ample effective divisor $H$ on $Y$ such that $\pi^* H \geq D$. Then, since $(Y, \Delta_Y, \fad)$ is globally $F$-regular, we know that for all $e \gg 0$, we have $1 \notin I_e ^{\Delta_Y, \fad} (H)$. By \autoref{eqn:Ieequality} again, we see that $1 \notin I_e ^{\Delta_X} (\pi^* H)$ for all $e \gg 0$. Since $\pi^* H \geq D$, we see that $D \notin I_e ^{\Delta_X} (L)$ as required.


For (b), we see immediately from \autoref{eqn:Ieequality} that for all $m \geq 0$ and all $e >0$, we have
\[ \dim_k H^0 (\cO_X (m\pi^*L_0))/ I_e ^{\Delta_X} (m\pi^*L_0) = \dim_k H^0 (\cO_Y (mL_0))/ I_e ^{\Delta_Y, \fad} (mL_0). \]
Then the claim follows from \autoref{def:FsigPair}
and \autoref{def:FsigTriple}.
\end{proof}




\begin{Cor} \label{Cor:numericallytrivial}
    Let $X$ be a globally $F$-regular projective variety $\pi:X\to Y$ be a projective birational map where $Y$ is normal. Suppose that there is a $\QQ$-divisor $\Delta_X\geq 0 $ such that $-K_X - \Delta_X \sim _{\bQ, \pi} 0  $ and $(X,\Delta_X)$ is globally $F$-regular. Set $\Delta_Y := \pi_*\Delta_X$. Then,\begin{itemize}
\item $K_Y+\Delta_Y$ is $\bQ$-Cartier and $K_X + \Delta_X = \pi^*(K_Y+\Delta_Y)$.

\item $(Y, \Delta_Y)$ is globally $F$-regular.

\item Moreover, for any ample divisor $L$ on $Y$, we have $\s (X,  \Delta_X; \pi^*L) = \s (Y, \Delta_Y; L)$
\end{itemize}
\end{Cor}

\begin{proof}
If the first claim holds, then we automatically have that $(Y,\Delta_Y)$ is globally $F$-regular by \cite{SchwedeSmithLogFanoVsGloballyFRegular}. Moreover, by \autoref{thm.GlobalFregularityProperBirational}, the equality of the $F$-signatures follows as well since the corresponding sequence of ideals sheaves will be trivial.

Hence, it suffices to show $K_Y+\Delta_Y$ is $\bQ$-Cartier and its pullback equals $K_X + \Delta_X$. Take $m>0$ such that $mK_X + m\Delta_X \sim \pi^*M$ where $M$ is a Cartier divisor on $Y$. Then, $\pi_*(mK_X + m\Delta_X)\sim M$ by the projection formula. Here, by construction, $\pi_*(K_X + \Delta_X) = K_Y+\Delta_Y$.  Hence, $K_Y+\Delta_Y$ is $\bQ$-Cartier. Now, let $F = K_X + \Delta_X - \pi^*(K_Y+\Delta_Y)$ be a divisor supported on the exceptional locus. Since $K_X + \Delta_X$ is assumed to be $\QQ$-linearly $\pi$-trivial, $F$ is a $\pi$-numerically trivial effective divisor. By the negativity lemma (which holds in positive characteristics by \cite[Section 2.3]{Birkar16}), we have $F = 0$. Thus, the claim follows.  
\end{proof}

\begin{proof}[Proof of \autoref{Thm:transformationRuleProperBirational}] Part (b) is a special case of \autoref{thm.GlobalFregularityProperBirational}(b) by setting $\Delta_X = \Delta_Y = 0$. Part (a) also follows from
\autoref{thm.GlobalFregularityProperBirational}(b),
with $\Delta_X=\Delta_Y=0$, since $K_{X|Y}\leq 0$
implies that the associated graded sequence
\[
\fa_m=\pi_*\cO_X(\lfloor-mK_{X|Y}\rfloor)=\cO_Y
\]
is trivial. Part (c) is a special case of \autoref{Cor:numericallytrivial} since $K_X + \Delta_X \sim_\QQ - \pi^* L $ is $\QQ$-linearly $\pi$-trivial. Thus, we have proved \autoref{Thm:transformationRuleProperBirational}.
\end{proof}

\section{Application to the positivity of limit $F$-signature.} \label{section:positivity}

In this section, we present an interesting application of our techniques. We fix the following notation throughout this section.

\begin{Not} \label{notation:spreadingout}
    For a normal projective variety $X$ over $\mathbb{C}$, we may consider its positive characteristic models $X_p$ for $p \gg 0$ by picking a suitable finitely generated $\mathbb{Z}$-algebra and a model $X_A$ of $X$ over $A$. Consider a closed point $s\in \Spec A$ and let $p $ be the characteristic of $k(s)$. A characteristic $p$ model of $X$ is any closed fiber $X_s$ of $X_A \to \Spec (A)$. Let $L$ be an ample line bundle on $X$. We will denote $L_s$ be the corresponding line bundle on $X_s$ on the characteristic $p$ models for $p \gg 0$.
\end{Not}

With notation as above, a normal projective variety $X$ is said to be of globally $F$-regular type if the characteristic $p$ reductions $X_s$ of $X$ is globally $F$-regular for a dense subset of $\Spec A$. A fundamental question about these models concerns the behavior of the $F$-signatures of the reductions $(X_s, L_s)$ for $p \gg 0$:

\begin{conjecture}{\cite{carvajal-rojas_fundamental_2016}} \label{conj:CRST}
With the above set up, assume that $X$ is of globally $F$-regular type. Then, after possibly replacing $A$ by a localization, there is a constant $C>0$ (not depending on $p$) such that for all $p \gg 0$ and any characteristic $p$ model $(X_s, L_s)$ of $(X,L)$, we have
\[ \s (X_s, L_s) \geq C . \]
\end{conjecture}

Our main observation is that given a projective variety $X$ of globally $F$-regular type over $\CC$, the above conjecture for one ample line bundle $L$ is equivalent to the conjecture for any other ample line bundle $L'$ over $X$.

\begin{thm}\label{thm:LowerBoundsSpreadout}
    Let $X$ be a projective variety of globally $F$-regular type over $\CC$. Suppose we have an ample divisor $L$ and a constant $C:= C(L)$ such that for all $p \gg 0$, we have the lower bound $ \s (X_s, L_s) \geq C $ for all $p \gg 0$. Then, for any ample divisor $L'$ on $X$, after possibly enlarging $A$, there exists a constant $C' := C' (L')$ such that $\s (X_s, L'_s) \geq C'$ for all closed points $s \in \Spec (A)$ of characteristic $p$ and all $p \gg 0$.
\end{thm}

\begin{lem} \label{lem:equivalentFsigandalpha}  \cite[Lemma 3.7]{LiuPandepositivitylimitfsignature}
Let $X$ be a globally $F$-regular type over $\CC$ and $L$ be an ample divisor on $X$. Then keeping \autoref{notation:spreadingout}, the following are equivalent after possibly replacing $A$ by a localization:

\begin{enumerate}
\item there exists a constant $C>0$ independent of $p$ such that $s(X_s, L_s) >C$ for all closed points $s \in \Spec (A)$ of characteristic $p$.

\item there exists a constant $D>0$ independent of $p$ such that $\alpha_F(X_s, L_s) >D$ for all closed points $s \in \Spec (A)$ of characteristic $p$.
\end{enumerate}
\end{lem}


\begin{proof}[Proof of \autoref{thm:LowerBoundsSpreadout}]
    First, working over $\CC$, we pick a constant $\lambda \in \mathbb{\QQ}_{>0}$ such that the line joining $\lambda L$ and  $ L'$ in the ample cone is effective (even ample). In other words, we write
    $\lambda L = L' + E$ for some effective Cartier divisor $E $.

    After possibly enlarging $A$, consider models $X_A, L_A, L'_A, E_A$ of $X, L, L', E$ respectively. Then we may assume that for all closed points $s \in \Spec (A)$, we have $E_s$ is an effective Cartier divisor and $\lambda L_s  = L'_s + E_s$.
    
Now, by \autoref{lem:equivalentFsigandalpha}, we know that there is a constant $D>0$ (not depending on $p$) such that for all closed points $s \in \Spec (A)$ with characteristic $p \gg 0$, we have $\FA (X_s, L_s) \geq D$. Then, using parts (c) and (d) of \autoref{thm:alphaonbigcone}
we have
\[ \FA (L' _s) \geq \FA(\lambda L_s ) \geq \frac{D}{\lambda} \]
for all closed points $s \in \Spec (A)$ with characteristic $p \gg 0$. Finally, since $\lambda$ is independent of $p$, using \autoref{lem:equivalentFsigandalpha} again, there exists a constant $C >0$ such that $\s(X_s, L'_s) \geq C$ for all closed points $s \in \Spec (A)$ of characteristic $p \gg 0$. This completes the proof of \autoref{thm:LowerBoundsSpreadout}.
\end{proof} 

\begin{eg}
    We may apply \autoref{thm:LowerBoundsSpreadout} to the case of cubic surfaces. For any cubic surface $X$ over $\CC$ with Du val singularities, the anti-canonical section ring $S(X, -K_X)$ defines a non-weakly exceptional three-dimensional KLT singularity, and hence by \cite{LiuPandepositivitylimitfsignature}, \autoref{conj:CRST} holds for $S(X, -K_X)$. Then, \autoref{thm:LowerBoundsSpreadout} allows us to conclude that \autoref{conj:CRST} holds for every section ring $S(X, L)$ of the cubic surface $X$ with respect to any ample divisor $L$ over $X$. We do not know any other direct method to \autoref{conj:CRST} for all section rings of cubic surfaces, since they may not support a non-weakly exceptional KLT singularity on the cone.
\end{eg}

\bibliographystyle{skalpha}
\bibliography{MainBib}

\end{document}